\documentclass[a4paper,12pt]{amsart}

\usepackage{amsmath,amscd,amsthm,amsxtra,amssymb}
\usepackage{epsfig,graphics,color,colortbl}
\usepackage{amssymb,latexsym}
\usepackage{mathrsfs}
\usepackage[poly,all]{xy}
\usepackage{marginnote}
\usepackage{xspace}
\usepackage[colorlinks=true, pdfstartview=FitV, linkcolor=blue,citecolor=blue,urlcolor=blue]{hyperref}

\usepackage{yfonts}
\usepackage{enumerate}

\usepackage[usenames,dvipsnames,svgnames,table]{xcolor}

\usepackage[normalem]{ulem}  

\allowdisplaybreaks[3]

\theoremstyle{plain}
\newtheorem{thm}{\bf Theorem}[section]
\newtheorem{df}[thm]{\bf Definition}
\newtheorem{prop}[thm]{\bf Proposition}
\newtheorem{cor}[thm]{\bf Corollary}
\newtheorem{lem}[thm]{\bf Lemma}
\newtheorem{conj}[thm]{\bf Conjecture}

\theoremstyle{definition}

\newtheorem{rem}[thm]{\bf Remark}

\newcommand{\nc}{\newcommand}
\nc{\Prop}{\begin{prop}}
\nc{\enprop}{\end{prop}}
\nc{\Lemma}{\begin{lem}}
\nc{\enlemma}{\end{lem}}

\newcommand{\C}{{\mathbb C}}
\newcommand{\Q}{\mathbb {Q}}
\newcommand{\Z}{{\mathbb Z}}
\newcommand{\R}{{\mathbb R}}

\newcommand{\seteq}{\mathbin{:=}}

\newcommand{\hd}{{\mathrm{hd}}}      					 
\newcommand{\soc}{\mathrm{soc}}					

\newcommand{\Sp}{\mathrm{span}_{\R_{\ge0}}}  	
\newcommand{\SpR}{\mathrm{span}_{\R}}  	

\newcommand{\g}{\mathfrak{g}}
\newcommand{\n}{\mathfrak{n}}

\newcommand{\Hom}{\operatorname{Hom}}
\newcommand{\HOM}{\mathrm{H{\scriptstyle OM}}}
\newcommand{\End}{\operatorname{End}}

\newcommand{\Mod}{\mbox{-$\mathrm{Mod}$}}
\newcommand{\gmod}{\mbox{-$\mathrm{mod}$}}

\newcommand{\proj}{\mbox{-$\mathrm{proj}$}}

\newcommand{\conv}{{\mathbin{\scalebox{1.1}{$\mspace{1.5mu}\circ\mspace{1.5mu}$}}}}
\newcommand{\hconv}{\mathbin{\scalebox{.9}{$\nabla$}}}
\newcommand{\sconv}{\mathbin{\scalebox{.9}{$\Delta$}}}

\newcommand{\Rm}{R^{\mathrm{norm}}}

\newcommand{\cmA}{\mathsf{A}}  
\newcommand{\wlP}{\mathsf{P}}   
\newcommand{\rlQ}{\mathsf{Q}}   
\newcommand{\weyl}{\mathsf{W}}  
\newcommand{\prD}{\Delta_+}            
\newcommand{\nrD}{\Delta_-}            
\newcommand{\Dn}{\Delta_{\mathfrak{n}}}            
\newcommand{\sg}{\mathfrak{S}}   
\newcommand{\Po}{\wlP}

\newcommand{\qQ}{\mathcal{Q}}
\newcommand{\bQ}{\overline{\qQ}}

\newcommand\Aq[1][{\g^+}]{A_q(#1)}
\newcommand\Aqq[1][{\g^+}]{A_{q}(#1)_{\Z[q,q^{-1}]}}

\newcommand{\Zq}{{\Z[q,q^{-1}]}}  		

\newcommand{\wt}{\mathrm{wt}} 		
\newcommand{\bR}{\mathbf{k}} 		
\nc{\corp}{\bR}
\newcommand{\catC}{ \mathscr{C}}  	
\newcommand{\dM}{ \mathsf{M }}              
\newcommand{\gW}{\mathsf{W}}
\newcommand{\sgW}{\mathsf{W}^*}
\newcommand{\tf}{{\widetilde{f}}}  		
\newcommand{\te}{{\widetilde{e}}}  		
\newcommand{\tF}{\widetilde{F}}  		
\newcommand{\tE}{\widetilde{E}}  		
\newcommand{\tEm}{\widetilde{E}^{\hskip 0.1em \rm max}}  		
\newcommand{\tEsm}{\widetilde{E}^{*{ \hskip 0.1em  \rm max}}}  		
\newcommand{\ep}{\varepsilon}  		
\newcommand{\ph}{\varphi}  		
\newcommand{\trivialM}{\mathsf{1}} 	
\newcommand{\rF}{\mathcal{T}} 		
\newcommand{\rFs}{\mathcal{T}^{\ms{1mu}*}}
\newcommand{\rW}{\mathbb{W}} 	
\newcommand{\cR}{{T}} 		
\newcommand{\cRs}{{T}^*} 		
\newcommand{\eC}{\mathcal{C}} 		
\newcommand{\lprec}{\prec_{\mathrm{lex}}} 			
\newcommand{\lpreceq}{\preceq_{\mathrm{lex}}} 		

\newcommand{\lsucceq}{\succeq_{\mathrm{lex}}}
\newcommand{\rprec}{\prec_{\mathrm{rev}}} 			
\newcommand{\rpreceq}{\preceq_{\mathrm{rev}}}

\newcommand{\bprec}{\prec_{\mathrm{bi}}} 			
\newcommand{\bpreceq}{\preceq_{\mathrm{bi}}}

\newcommand{\wuprec}[1]{\prec^{\underline{#1} \hskip 0.1em }} 			
\newcommand{\wupreceq}[1]{\preceq^{\underline{#1} \hskip 0.1em }}

\newcommand{\wpreceq}[1]{\preceq^{#1 \hskip 0.1em }} 		

\newcommand{\wsucceq}[1]{\succeq^{#1 \hskip 0.1em }}
\newcommand{\Ht}{\mathrm{ht}} 		
\newcommand{\cd}{\mathfrak{d}} 		
\newcommand{\cA}{\mathsf{A}} 		 
\newcommand{\La}{\Lambda} 			
\newcommand{\tLa}{\widetilde{\Lambda}} 			
\newcommand{\Dd}{\text{ \textfrak{d}}} 			
\newcommand{\Res}{\mathrm{Res}} 			
\newcommand{\ch}{{\mathrm{ch}_q}} 			
\newcommand{\gB}{{\mathbf{B}^{\rm up}}} 			
\newcommand{\gG}{{G^{\rm up}}} 			
\newcommand{\iR}{{{_i}R}} 			 
\newcommand{\Ri}{{R_i}} 			 
\newcommand{\iB}{{{_i}B(\infty)}} 			 
\newcommand{\Bi}{{B_i (\infty)}} 			 
\newcommand{\Oint}{{\mathcal{O}_{\rm int}}}
\newcommand{\Ointr}{{\mathcal{O}_{\rm int}^{\ms{3mu}\rm r}}}

\nc{\be}{\begin{enumerate}}
\newcommand{\bnum}{\be[{\rm(i)}]}
\newcommand{\bna}{\be[{\rm(a)}]}

\newcommand{\rtl}{\rlQ}

\newcommand{\rmat}[1]{{\mathbf{r}}_%
{\mspace{-2mu}\raisebox{-.6ex}{${\scriptstyle{#1}}$}}}

\nc{\ms}{\mspace}
\nc{\cl}{\colon}
\nc{\ro}{{\rm (}}
\nc{\rf}{{\rm )}}
\nc{\noi}{\noindent}
\nc{\bl}{\bigl(}
\nc{\br}{\bigr)}

\newenvironment{myequation}
{\relax\setlength{\arraycolsep}{1pt}\begin{eqnarray}}
{\end{eqnarray}}
\newenvironment{myequationn}
{\relax\setlength{\arraycolsep}{1pt}\begin{eqnarray*}}
{\end{eqnarray*}}

\nc{\eq}{\begin{myequation}}
\nc{\eneq}{\end{myequation}}
\nc{\eqn}{\begin{myequationn}}
\nc{\eneqn}{\end{myequationn}}

\newenvironment{myarray}[1]{\relax\setlength{\arraycolsep}{1pt}
\begin{array}{#1}}{\end{array}\relax}

\newcommand{\ba}{\begin{myarray}}
\newcommand{\ea}{\end{myarray}}

\nc{\hs}{\hspace*}
\nc{\set}[2]{\left\{{#1}\mid{#2}\right\}}
\nc{\snoi}{\smallskip\noi}
\nc{\al}{\alpha}
\nc{\rmz}{\setminus\{0\}}
\nc{\tens}[1][]{\mathbin{\otimes_{\raise1.5ex\hbox to-.1em{}#1}}}
\nc{\vphi}{\varphi}
\nc{\ee}{\end{enumerate}}
\nc{\la}{\lambda}
\nc{\bc}{\begin{cases}}
\nc{\ec}{\end{cases}}
\nc{\qtq}[1][and]{\quad\text{#1}\quad}
\nc{\qt}[1]{\quad\text{#1}}
\nc{\dual}{{\displaystyle{\ms{1mu}\star}}}
\nc{\wle}{\preceq}
\nc{\epito}{\twoheadrightarrow}
\nc{\epiTo}[1][]{\underset{#1}{\twoheadrightarrow}}
\nc{\Proof}{\begin{proof}}
\nc{\lan}{\langle}
\nc{\ran}{\rangle}
\nc{\ang}[1]{\lan{#1}\ran}
\nc{\QED}{\end{proof}}
\nc{\soplus}{\scalebox{.65}{\raisebox{.2ex}{$\displaystyle\bigoplus$}}}
\nc{\eps}{\varepsilon}
\nc{\on}{\operatorname}
\nc{\supp}{\on{supp}}
\nc{\sct}{strongly commute\xspace}
\nc{\scts}{strongly commutes\xspace}

\numberwithin{equation}{section}

\begin{document}

\title[Monoidal categories associated with strata of flag manifolds]
{Monoidal categories associated with strata of flag manifolds}

\author[Masaki Kashiwara]{Masaki Kashiwara}
\thanks{The research of M.Ka.\
was supported by Grant-in-Aid for Scientific Research (B)
15H03608, Japan Society for the Promotion of Science.}
\address[Masaki Kashiwara]{Research Institute for Mathematical Sciences, Kyoto University,
Kyoto 606-8502, Japan \& Korea Institute for Advanced Study, Seoul 02455, Korea }
\email[Masaki Kashiwara]{masaki@kurims.kyoto-u.ac.jp}

\author[Myungho Kim]{Myungho Kim}
\address[Myungho Kim]{Department of Mathematics, Kyung Hee University, Seoul 02447, Korea}
\email[Myungho Kim]{mkim@khu.ac.kr}
\thanks{The research of M. Kim was supported by the National Research Foundation of
Korea(NRF) Grant funded by the Korea government(MSIP) (NRF-2017R1C1B2007824).}

\author[Se-jin Oh]{Se-jin Oh}
\thanks{ The research of S.-j. Oh was supported by the National Research Foundation of
Korea(NRF) Grant funded by the Korea government(MSIP) (NRF-2016R1C1B2013135).}
\address[Se-jin Oh]{Department of Mathematics, Ewha Womans University, Seoul 120-750, Korea}
\email[Se-jin Oh]{sejin092@gmail.com}

\author[Euiyong Park]{Euiyong Park}
\thanks{The research of E.P.\ was supported by the National Research Foundation of Korea(NRF) Grant funded by the Korean Government(MSIP)(NRF-2017R1A1A1A05001058). }
\address[Euiyong Park]{Department of Mathematics, University of Seoul, Seoul 02504, Korea}
\email[Euiyong Park]{epark@uos.ac.kr}




\maketitle

\begin{abstract}
We construct a monoidal category $\catC_{w,v}$
 which categorifies
the doubly-invariant algebra  $^{N'(w)}\C[N]^{N(v)}$
associated with Weyl group elements $w$ and $v$.
It gives, after a localization, the coordinate algebra
$\C[\mathcal{R}_{w,v}]$ of the open Richardson variety associated with $w$ and $v$.
The category $\catC_{w,v}$ is realized
as a subcategory of the graded module category of
a quiver Hecke algebra $R$.
When $v= \mathrm{id}$, $\catC_{w,v}$ is the same as the monoidal category which provides a monoidal categorification of the quantum unipotent coordinate algebra $\Aqq[\n(w)]$ given by Kang-Kashiwara-Kim-Oh.
We show that the category $\catC_{w,v}$ contains special determinantial modules $\dM(w_{\le k}\Lambda, v_{\le k}\Lambda)$ for $k=1, \ldots, \ell(w)$,
which commute with each other. When the quiver Hecke algebra $R$ is symmetric, we find a formula of the degree of $R$-matrices between the determinantial modules $\dM(w_{\le k}\Lambda, v_{\le k}\Lambda)$.
 When it is of finite $ADE$ type,
we further prove that there is an equivalence of categories between $\catC_{w,v}$ and $\catC_u$ for $w,u,v \in \weyl$ with $w = vu$ and $\ell(w) = \ell(v) + \ell(u)$.

\end{abstract}

\tableofcontents

\section*{Introduction}

A \emph{quantum cluster algebra} is a non-commutative $q$-deformation of a \emph{cluster algebra} which is a special $\Z$-subalgebra of a rational function field with
a set of generators grouped into overlapping subsets, called \emph{clusters}, and a procedure which makes new clusters, called \emph{mutation}.
The (quantum) cluster algebras appear naturally in various studies of mathematics including representation theory of Kac-Moody algebras.
Let $\g$ be a Kac-Moody algebra of symmetric type and $U_q(\g)$ its quantum group. We take an element $w$ of the Weyl group $\weyl$. It was shown in \cite{GLS11} that
the quantum unipotent coordinate algebra $\Aq[\n(w)]$ associated with $U_q(\g)$ and $w$ has a quantum cluster algebra structure.
It turned out that the intersection of the upper global basis and $\Aq[\n(w)]$ is a basis of $\Aq[\n(w)]$ (\cite{Kimura12}).
On the other hand, it was proved that there is a cluster algebra structure inside the coordinate ring of a \emph{Richardson variety}.
For $w,v\in \weyl$ with $v \le w$, one can consider the \emph{open Richardson variety} $\mathcal{R}_{w,v}$ which is
the intersection of the Schubert cell and the opposite Schubert cell attached to $w$ and $v$ respectively.
It was shown in \cite{Lec15} that, when $\g$ is of finite $ADE$ type, the coordinate algebra $\C[\mathcal{R}_{w,v}]$ contains a subalgebra with the cluster algebra structure
coming from a certain subcategory of module category of a preprojective algebra which is determined by $w,v$.
The cluster variables of the initial cluster are given by the irreducible factors of
the generalized minors determined by $w_{\le j}$ and $v_{\le j}$ for $j \in J_{\underline{w}, v}$ (for these notations, see \eqref{eq:Jwv}).
A quantization of the coordinate algebra $\C[\mathcal{R}_{w,v}]$ was studied in \cite{LenYak16} in the aspect of quantum cluster algebras.

The \emph{quiver Hecke algebras} (or \emph{Khovanov-Lauda-Rouquier algebras}) were introduced to categorify the half of a quantum group (\cite{KL09, KL11, R08}).
The algebras have special quotients, called \emph{cyclotomic quiver Hecke algebras}, which provide a categorification of irreducible integrable highest weight modules (\cite{KK11}).
It was shown in \cite{BK09, R08} that cyclotomic quiver Hecke algebras of affine type $A$ are isomorphic to cyclotomic Hecke algebras. In that sense, the quiver Hecke algebras are a vast generalization of
Hecke algebra in the direction of categorification.
When the quiver Hecke algebra is \emph{symmetric}  and the base field is of characteristic zero,
 the upper global basis corresponds to the set of isomorphism classes of simple modules (\cite{R11, VV09}).

Let $R(\beta)$ be the symmetric quiver Hecke algebra corresponding to a quantum group $U_q(\g)$ of symmetric type and denote by $R\gmod$ the category of graded $R$-modules which are finite-dimensional over the base field $\corp$.
It was proved in \cite{KKKO17} that there is a monoidal subcategory $\mathcal{C}_w$ of $R\gmod$ which provides
a monoidal categorification of $\Aq[\n(w)]$ as a quantum cluster algebra.
The \emph{determinantial modules} $\dM(\lambda, \mu)$ are a key ingredient for the monoidal categorification.
They are simple $R$-modules determined by a dominant integral weight $\Lambda$ and $\lambda, \mu \in \weyl \Lambda$ with $\lambda \le \mu$, which correspond to the unipotent quantum minors
under the categorification.
The determinantial modules defined by a reduced expression $\underline{w}$ and $R$-matrices among them give a quantum monoidal seed in $\mathcal{C}_w$.
This monoidal categorification gave a proof of the conjecture that any cluster monomial is a member of the upper global basis up to a power of $q^{1/2}$.

In this paper, we construct a monoidal category $\catC_{w,v}$ associated with Weyl group elements $w$ and $v$ as a subcategory of $R\gmod$ for an \emph{arbitrary} quiver Hecke algebra.
The results on convex orders and cuspidal modules obtained in \cite{TW16}  play a key role in proving properties on $\catC_{w,v}$.
Let $\prD$ (resp.\ $\nrD$) be the set of positive roots (resp.\ negative roots) and set
$ \rlQ_+ \seteq\sum_{i\in I} \Z_{\ge 0} \alpha_i$ and $\rlQ_- \seteq - \rlQ_+$.
For an $R(\beta)$-module $M$, we define
\begin{align*}
\gW(M) &:= \{  \gamma \in  \rlQ_+ \cap (\beta - \rlQ_+)  \mid  e(\gamma, \beta-\gamma) M \ne 0  \}, \\
\sgW(M) &:= \{  \gamma \in  \rlQ_+ \cap (\beta - \rlQ_+)  \mid  e(\beta-\gamma, \gamma) M \ne 0  \}.
\end{align*}
For $w, v \in \weyl$, let us denote by $\catC_{w, v}$ the  full subcategory of $R\gmod$
whose objects $M$ satisfy
\begin{align*}
\gW(M) \subset \rtl_+\cap w\rtl_- \quad \text{ and } \quad \sgW(M) \subset \rtl_+\cap v \rtl_+.
\end{align*}
The category $\catC_{w,v}$ is stable under taking subquotients, extensions, convolution products and grading shifts (Proposition \ref{Prop: closedness}).
When $v= \mathrm{id}$, $\catC_{w,v}$ is  nothing but  the monoidal category $\mathcal{C}_{w^{-1}}$ introduced in \cite{KKKO17}.
We also give a membership condition on $\catC_{w,v}$ using the cuspidal decomposition in Proposition \ref{Prop: membership}.
For $w, v \in \weyl$, we define $ A_{w, v}$ to be
the $\Zq$-linear subspace of $\Aqq[\n]$ consisting of all elements $x \in \Aqq[\n]$ such that
\begin{align*}
e_{i_1} \cdots e_{i_l}x=0 \quad \text{ and } \quad e_{j_1}^* \cdots e_{j_{l'}}^*x=0
\end{align*}
for any $(i_1, \ldots, i_l) \in I^\beta$ with $\beta \in \rlQ_+ \cap w\rlQ_+ \setminus \{ 0 \} $ and any $(j_1, \ldots, j_{l'}) \in I^\gamma$ with $\gamma \in \rlQ_+ \cap v\rlQ_- \setminus \{ 0 \} $.
Theorem \ref{Thm: A and catC} tells that the Grothendieck ring of $\catC_{w,v}$ is equal to $A_{w,v}$ under the categorification, which implies that $A_{w,v}$ forms a subalgebra of $\Aqq[\n]$.
The subalgebra $A_{w,v}$ can be understood as a $q$-deformation of the doubly-invariant algebra $^{N'(w)} \C[N] ^{N(v)}$, where
$N$ is the unipotent radical, $N'(w) = N \cap (w N w^{-1})$ and $ N(v) = N \cap (v N^- v^{-1})$ (Remark \ref{Rmk: geo}).

The category $\catC_{w,v}$ contains special determinantial modules related to $w$ and $v$.
Fix a reduced expression $\underline{w}$ of $w$ and define $w_{\le k}$ and $v_{\le k}$ by $\eqref{Eq: def of wk}$ and $\eqref{Eq: def of wk vk}$ for $k=1, \ldots, \ell(w)$.
It is proved in Proposition~\ref{prop: M(wk,vk) and Cwv} that
the determinantial modules $\dM(w_{\le k}\Lambda, v_{\le k}\Lambda)$ are contained in $\catC_{w,v}$ for any dominant integral weight $\Lambda$.
We then show that the determinantial modules $\dM(w_{\le k}\Lambda, v_{\le k}\Lambda)$ commute with  each other (Theorem \ref{Thm: commutation for dM}).
When the quiver Hecke algebra $R$ is symmetric,
 we compute the degree of $R$-matrices between them
in Theorem \ref{Thm: degree of R} as follows:
\begin{align*}
\La( \dM(w_{\le j}\Lambda, v_{\le j}\Lambda),  \dM(w_{\le k}\Lambda', v_{\le k}\Lambda')) = ( w_{\le j} \Lambda + v_{\le j}\Lambda, v_{\le k}\Lambda' - w_{\le k}\Lambda' )
\end{align*}
for $1\le k<j\le \ell$.
 When  specialized  at $q=1$, the determinantial modules $ \allowbreak \dM(w_{\le k}\Lambda, v_{\le k}\Lambda)$
correspond to the generalized minors  associated with
$w_{\le k}$ and $v_{\le k}$ which provide the initial cluster in  the coordinate ring $\C[\mathcal{R}_{w,v}]$ of the open Richardson variety in \cite{Lec15}.
 Note that the coordinate ring $\C[\mathcal{R}_{w,v}]$  is  a localization of the doubly-invariant algebra  $^{N'(w)} \C[N] ^{N(v)}$  by  a subset of
such generalized minors (\cite[Theorem 2.12]{Lec15}).
It is expected that our category $\catC_{w,v}$ gives a monoidal categorification of a quantization of the  cluster
algebra arising from $\mathcal{R}_{w,v}$ given in \cite{Lec15}.

Suppose that $\g$ is of finite $ADE$ type.
Then S.\ Kato (\cite{Kato14}) constructed the Saito reflection functors
\begin{align*}
\rF_i \cl \Ri (\beta) \gmod \buildrel \sim \over \longrightarrow  \iR (s_i\beta) \gmod, \quad
\rFs_i \cl \iR (\beta) \gmod \buildrel \sim \over \longrightarrow  \Ri (s_i\beta) \gmod,
\end{align*}
which are equivalences of categories
(see \S\,\ref{Sec: ADE}).
 We show that
they induce equivalences of categories (Theorem \ref{Thm: equi Cwv}):
\begin{align*}
\rF_i|_{\catC_{s_iw, s_iv}} \cl \catC_{s_i w, s_i v} \buildrel \sim \over \longrightarrow \catC_{ w, v}, \qquad
\rFs_i|_{\catC_{w, v}} \cl \catC_{ w,  v} \buildrel \sim \over\longrightarrow \catC_{s_i w, s_i v}
\end{align*}
for $w,v\in \weyl$ such that $s_iw > w$ and $s_i v > v$,
Thus, when $w,u,v \in \weyl$ such that $w = vu$ and $\ell(w) = \ell(v) + \ell(u)$, there is a category equivalence between $\catC_{w,v}$ and $\catC_u$.
After finishing the paper, it was proved by S.~Kato (\cite{Kato17}) and P.~J.~McNamara (\cite{Mc17})  independently that 
the reflection functors are monoidal, which implies that Conjecture \ref{Conj: monoidal categorification} is true (Remark \ref{Rmk: Conj}).

\medskip

The paper is organized as follows.
In Section \ref{Sec: Preliminaries}, we review quantum groups, quiver Hecke algebras and results on the convex preorders given in \cite{TW16}.
We also prove several lemmas on properties of convex preorders. In Section \ref{Sec: Cwv},
we construct and investigate the categories $\catC_{w,v}$ and the algebras $A_{w,v}$, and prove the main theorem for $\catC_{w,v}$ and $A_{w,v}$.
Section \ref{Sec: R-matrix} contains a review on  $R$-matrices and new lemmas for proving the formula of the degree of $R$-matrices between the determinantial modules.
In Section \ref{Sec: determinatial module}, we show the commuting property on $\dM(w_{\le k}\Lambda, v_{\le k}\Lambda)$ and give a formula to compute the degree
$\La( \dM(w\Lambda, v\Lambda),  \dM(w_{\le k}\Lambda', v_{\le k}\Lambda'))$ of $R$-matrices when the quiver Hecke algebra is symmetric.
Section \ref{Sec: ADE} specializes to the case of finite $ADE$ types. We show that there is a
category equivalence between
$ \catC_{s_i w, s_i v} $ and $ \catC_{ w, v}$ for $w,v\in \weyl$ with $s_iw > w$ and $s_i v > v$,
which yields a category equivalence between
between $\catC_{w,v}$ and $\catC_u$ when $w,u,v \in \weyl$ such that $w = vu$ and $\ell(w) = \ell(v) + \ell(u)$.

\medskip
{\bf Acknowledgments}
We thank P.\ Tingly and B.\ Webster for their kindness to
explain us their works on convex preorders.
We also thank B.\ Leclerc for many discussions around this subject.
We thank Y.\ Kimura for fruitful communication.
Finally, we want to thank P.\ McNamara for discussions on the PBW bases.

\vskip 2em

\section{Preliminaries} \label{Sec: Preliminaries}

\subsection{Quantum groups}
Let $I$ be an index set.
A {\it Cartan datum} $ (\cmA,\wlP,\Pi,\Pi^\vee,(\cdot,\cdot) ) $
consists of
\begin{enumerate}
\item a free abelian group $\wlP$, called the {\em weight lattice},
\item $\Pi = \{ \alpha_i \mid i\in I \} \subset \wlP$,
called the set of {\em simple roots},
\item $\Pi^{\vee} = \{ h_i \mid i\in I \} \subset \wlP^{\vee}\seteq
\Hom( \wlP, \Z )$, called the set of {\em simple coroots},
\item a $\Q$-valued symmetric bilinear form $(\cdot,\cdot)$ on $\Po$ which satisfies
\begin{enumerate}
\item  $(\alpha_i,\alpha_i)\in 2\Z_{>0}$ for $i\in I$,
\item $\langle h_i, \lambda \rangle =\dfrac{2(\alpha_i,\lambda)}{(\alpha_i,\alpha_i)}$ for $i\in I$ and $\lambda \in \Po$,
\item $\cmA \seteq (\langle h_i,\alpha_j\rangle)_{i,j\in I}$ is
a {\em generalized Cartan matrix}, i.e.,
$\langle h_i,\alpha_i\rangle=2$ for any $i\in I$ and
$\langle h_i,\alpha_j\rangle \in\Z_{\le0}$ if $i\not=j$,
\item $\Pi$ is a linearly independent set,
\item for each $i\in I$, there exists $\Lambda_i \in \wlP$
such that $\langle h_j, \Lambda_i \rangle = \delta_{ij}$ for any $j\in I$.
\end{enumerate}
\end{enumerate}
We denote by $\prD$ the set of positive roots and set
$$
 \rlQ = \bigoplus_{i \in I} \Z \alpha_i, \quad \rlQ_+ = \sum_{i\in I} \Z_{\ge 0} \alpha_i, \quad \rlQ_- = \sum_{i\in I} \Z_{\le 0} \alpha_i.
$$
We write $\Ht (\beta)=\sum_{i \in I} k_i$
and $\supp(\beta)=\set{i\in I}{k_i\not=0}$  for $\beta=\sum_{i \in I} k_i \alpha_i \in \rlQ_+$.

The {\em Weyl group} $\weyl$ associated with the Cartan datum
is the subgroup of $\mathrm{Aut}(\wlP)$ generated by
$\{s_i\}_{i\in I}$ given by
$$s_i(\lambda)=\lambda-\langle h_i, \lambda\rangle\alpha_i \qquad
\text{for $\lambda\in \wlP$.}$$
For $w \in \weyl$, the expression $w = s_{i_1} \cdots s_{i_\ell}$ is called a {\it reduced expression} if $\ell$ is minimal among all such expressions, and we set $\ell(w) \seteq \ell$.
For $w,v \in \weyl$, we write $w \ge v$ if there is a reduced expression of $v$ which appears in a reduced expression $w$ as a subexpression.

Let $U_q(\g)$ be the quantum group associated with the Cartan datum, which is an associative algebra over $\Q(q)$ generated by $e_i$, $f_i$ $(i\in I)$ and $q^h$
$(h\in \wlP^\vee)$ with certain defining relations (see \cite[Chap.\ 3]{HK02} for details).
For $\beta \in \rlQ$,  $U_q(\g)_\beta$ is the weight space of $U_q(\g)$ with weight $\beta$ and we set $\wt(x) = \beta$ for $x \in U_q(\g)_\beta$.
We denote by $U_q^+(\g)$ (resp.\ $U_q^-(\g)$) the subalgebra of $U_q(\g)$ generated by $e_i $ (resp.\ $f_i $) for $i\in I$.
For $n\in \Z_{\ge 0}$ and $i \in I$, we set $e_i^{(n)} \seteq e_i^n / [n]_i!$ and $f_i^{(n)} \seteq f_i^n / [n]_i!$,
where  $q_i = q^{ (\alpha_i, \alpha_i)/2 }$,
\begin{align*}
  [n]_i =\frac{ q^n_{i} - q^{-n}_{i} }{ q_{i} - q^{-1}_{i} } \quad  \text{ and } \quad
 [n]_i! = \prod^{n}_{k=1} [k]_i.
 \end{align*}
We denote by $U_\Zq^-(\g)$ the $\Zq$-subalgebra of $U_q^-(\g)$ generated by $f_i^{(n)}$ for $i\in I$ and $n \in \Z_{\ge0}$.
 Similarly we define $U_\Zq^+(\g)$.

The definitions of the category $\Oint(\g)$ of  integrable $U_q(\g)$-modules and crystal bases can be found in \cite{HK02, K91}.
For a left $U_q(\g)$-module $M$, we denote by $M^r$ the right $U_q(\g)$-module $\{ m^r \mid m \in M \}$ with the right actions defined by
\[
(m^r) x = (\varphi(x) m)^r \quad \text{ for $m\in M$ and $x \in U_q(\g)$,}
\]
where $\varphi$ is the $\Q(q)$-antiautomorphism of $U_q(\g)$ given by
$\varphi (e_i) = f_i$, $\varphi (f_i) = e_i$, $ \varphi(q^h) = q^h$.
We denote by $ \Ointr(\g)$ the category of right integrable modules $M^r$ such that $M \in \Oint(\g)$.

For a dominant integral weight $\Lambda \in \wlP_+$,
we  denote by $V_q(\Lambda)$  the irreducible highest weight module with highest weight $\Lambda$.
There is a unique non-degenerate symmetric bilinear form $(\ ,\ )$ such that
\[
 \text{$(u_\Lambda, u_\Lambda)=1$ and $(xu, v) = (u, \varphi(x) v)$ for $u,v \in V_q(\Lambda)$ and $x\in U_q(\g)$, }
\]
where $u_\Lambda$ is the highest weight vector of $V_q(\Lambda)$. This bilinear form induces the non-degenerated bilinear form
\[
\langle\ , \ \rangle:  V_q(\Lambda) ^r \otimes_{U_q(\g)} V_q(\Lambda)  \longrightarrow \Q(q)
\]
defined by $\langle u^r, v \rangle = (u,v)$ for $u,v \in V_q(\g)$.

It was known that every integrable $U_q(\g)$-module in $\Oint(\g)$ has a crystal base. In particular, the irreducible highest weight modules $V_q(\Lambda)$ for $\Lambda \in \wlP_+$
have crystal bases. Let us recall the notion of a crystal.

\begin{df}
A \emph{crystal} is a set $B$ together with  maps $\wt: B \rightarrow \wlP$, $\ph_i$, $\ep_i : B \rightarrow \Z \cup \{ \infty \}$,
and $\te_i$, $\tf_i : B \rightarrow B\cup\{0\}$ which satisfy the following:
\bna
\item $\ph_i(b) = \ep_i(b) + \langle h_i, \wt(b) \rangle$,
\item $\wt(\te_i b) = \wt(b) + \alpha_i$, $\wt(\tf_i b) = \wt(b) - \alpha_i$ if $\te_i b$ and $\tf_i b\in B$,
\item for $b,b' \in B$ and $i\in I$, $b' = \te_i b$ if and only if $b = \tf_i b'$,
\item for $b \in B$, if $\ph_i(b) = -\infty$, then $\te_i b = \tf_i b = 0$,
\item if $b\in B$ and $\te_i b \in B$, then $\ep_i(\te_i b) = \ep_i(b) - 1$ and $ \ph_i(\te_i b) = \ph_i(b) + 1$,
\item if $b\in B$ and $\tf_i b \in B$, then $\ep_i(\tf_i b) = \ep_i(b) + 1$ and $ \ph_i(\tf_i b) = \ph_i(b) - 1$.
\end{enumerate}
\end{df}
We denote by $B(\infty)$ the crystal of $U_q^-(\g)$ in which $\ep_i$ and $\ph_i$ are given as
\[
\ep_i(b) = \max\{ k \ge 0 \mid \te_i^k b \ne 0 \}, \qquad
\ph_i(b) = \ep_i(b) + \langle h_i, \wt(b) \rangle.
\]
The $\Q(q)$-antiautomorphism $*$ of $U_q(\g)$ defined by
\[
e_i^* = e_i,\qquad f_i^* = f_i, \qquad {(q^h)}^* = q^{-h}
\]
gives another crystal structure with $ \te_i^*$, $ \tf_i^*$, $\ep_i^*$, $\ph_i^*$  on $U_q^-(\g)$.
For $b \in B(\infty)$ and $i\in I$, we set
\[
\te_i ^{\hskip 0.2em  \rm max} (b) = \te_i ^{\ep_i(b)} (b), \qquad
\te_i ^{* \hskip 0.2em  \rm max} (b) = \te_i ^{* \hskip 0.2em \ep_i^*(b)} (b).
\]

We define the \emph{quantum coordinate ring} $\Aq[\g]$ by
\[
\Aq[\g] = \{ u \in U_q(\g)^* \mid \text{ $U_q(\g)u$ belongs to $\Oint(\g)$ and $uU_q(\g)$ belongs to $\Ointr(\g)$ }  \},
\]
where $ U_q(\g)^* = \Hom_{\Q(q)} ( U_q(\g), \Q(q) )$. Then we have the following.
\begin{prop}[{\cite[Proposition 7.2.2]{K93}}]
There is an isomorphism $\Phi$ of $U_q(\g)$-bimodules
\[
\Phi: \bigoplus_{\Lambda \in \wlP_+} V_q(\Lambda) \otimes_{\Q(q)} V_q(\Lambda)^r \buildrel \sim\over \longrightarrow \Aq[\g]
\]
defined by $\Phi(u \otimes v^r) (x) = \langle v^r, xu \rangle$ for $v,u \in V_q(\Lambda)$ and $x \in U_q(g)$.
\end{prop}

The tensor product $U_q^+(\g) \otimes U_q^+(\g)$ has the algebra structure defined by
\[
(x_1 \otimes x_2) \cdot (y_1 \otimes y_2) = q^{- ( \wt(x_2), \wt(y_1) )}  (x_1y_1 \otimes x_2y_2)
\]
for homogeneous elements $x_1, x_2, y_1, y_2 \in U_q^+(\g)$.
We define the $\Q(q)$-algebra homomorphism $\Dn \cl U_q^+(\g)\longrightarrow U_q^+(\g) \otimes U_q^+(\g)$ by
\[
\Dn (e_i) = e_i \otimes 1 + 1 \otimes e_i \qquad \text{ for } i\in I.
\]
We set
\[
\Aq[\n] = \bigoplus_{\beta \in \rlQ_-} \Aq[\n]_\beta, \qquad \text{ where $\Aq[\n]_\beta\seteq \Hom_{\Q(q)}(U_q^+(\g)_{-\beta}, \Q(q))$.}
\]
Defining the bilinear from $\langle \cdot, \cdot \rangle: ( \Aq[\n] \otimes \Aq[\n] ) \times ( U_q^+(\g) \otimes U_q^+(\g) )$ by
\[
\langle \psi \otimes \theta, x \otimes y \rangle = \theta(x) \psi(y),
\]
 the multiplication on $\Aq[\n]$ is given by
\[
(\psi \cdot \theta )(x) = \langle \psi \otimes \theta, \Dn(x) \rangle \quad \text{ for } \psi, \theta \in \Aq[\n],\ x \in U_q^+(\g).
\]
This is called the \emph{unipotent quantum coordinate ring}.
 Note that there is a $\Q(q)$-algebra
isomorphism between $A_q(\n)$ and $U_q^-(\g)$ (\cite[Lemma 8.2.2]{KKKO17}).

We set
\[
\Aqq[\n] = \{ \psi \in \Aq[\n] \mid \psi( U_{\Zq}^+(\g) ) \subset \Zq  \}.
\]

 Let $e_i$ and $e_i^*\in\End(A_q(\n))$ denote
the endomorphisms induced by the right multiplication and the left multiplication of $e_i$ on $U_q^+(\g)$, respectively.

\begin{lem}[{\cite[Lemma 8.2.1]{KKKO17}}] \label{Lem: boson}
For $u,v \in A_q(\n)$, we have
\begin{align*}
e_i(uv) = (e_i u)v + q^{(\alpha_i, \wt(u))}u(e_i v) \quad \text{ and }\quad e_i^*(uv) = u(e_i^* v) + q^{(\alpha_i, \wt(v))}(e_i^* u)v.
\end{align*}
\end{lem}

The algebra $\Aq[\n]$ has the upper global basis
\[
\gB(\Aq[\n]) = \{ \gG(b) \mid b  \in B(\Aq[\n])  \},
\]
where $B(\Aq[\n])$ is the crystal basis  of $\Aq[\n]$ which is isomorphic to $B(\infty).$

\begin{lem}[{\cite[Lemma 5.1.1]{K93}}] \label{Lem: upper global 1}
For $b \in B(\Aq[\n]) $, we have
\begin{align*}
 e_i^{(m)} \gG(b) =   \gG(\tilde{e}_i^{m}b), \qquad  {e_i^{* \hskip 0.1em (n)} } \gG(b) =   \gG(\tilde{e}_i^{*\hskip 0.1em  n}b),
\end{align*}
where $m = \ep_i(b)$ and $n = \ep_i^*(b)$.
\end{lem}

The homomorphism $p_\n: \Aq[\g] \rightarrow \Aq[\n]$ induced by $U_q^+(\g) \rightarrow U_q(\g)$ is given by
\[
\langle p_\n(\psi), x  \rangle = \psi(x) \quad \text{ for any $x \in U_q^+(\g)$}.
\]
\begin{df}
For $\Lambda \in \wlP_+$ and $\mu, \zeta \in \weyl \Lambda$, we define a \ro generalized\/\rf\ \emph{quantum minor} $\Delta(\mu, \zeta)$ and a
\emph{unipotent quantum minor} $D(\mu, \zeta)$ as follows:
\begin{align*}
\Delta(\mu, \zeta) \seteq \Phi(u_\mu \otimes u_\zeta^r) \in \Aq[\g], \quad D(\mu, \zeta) \seteq p_\n (\Delta(\mu, \zeta)) \in \Aq[\n].
\end{align*}
\end{df}

 Recall that for $\la$, $\mu\in\wlP$, we write
$\la\wle\mu$ if there exists a sequence of
real positive roots $\beta_k$ ($1\le k\le \ell$)
such that $\la=s_{\beta_\ell}\cdots s_{\beta_1}\mu$ and
$(\beta_k,s_{\beta_{k-1}}\cdots s_{\beta_1}\mu)>0$ for $1\le k\le\ell$.
When $\La\in\wlP_+$ and $\la$, $\mu\in \weyl\La$,
the relation $\la\wle\mu$ holds if and only if there exist
$w$, $v\in\weyl$ such that
$\la=w\La$, $\mu=v\La$ and $v\le w$.

\begin{lem}[{\cite[Lemma 9.1.1, Corollary 9.1.3, Lemma 9.1.4]{KKKO17}}]
\label{Lem: minor}
\bnum
\item For $\Lambda \in \wlP_+$ and $\mu, \zeta \in \weyl \Lambda$, $D(\mu, \zeta)$ is a member of the upper global basis of $\Aq[\n]$ if $\mu \wle \zeta$. Otherwise, $D(\mu, \zeta)$ is zero.
\item Let $\Lambda, \Lambda' \in \wlP_+$ and $w,v \in \weyl$ with $v \le w$. Then
\[
D(w\Lambda, v\Lambda)D(w\Lambda', v\Lambda') = q^{-( v\Lambda, v\Lambda' - w\Lambda' )} D(w(\Lambda+\Lambda'), v(\Lambda+\Lambda')).
\]
\end{enumerate}
\end{lem}

\begin{lem}[{\cite[Lemma 9.1.5]{KKKO17}}] \label{Lem: minor E}
Let $\Lambda \in \wlP_+$ and $\mu,\zeta \in \weyl \Lambda$ with $\mu \wle \zeta$ and $i\in I$.
\bnum
\item If $n \seteq \langle h_i, \mu\rangle \ge 0$, then
\[
\text{ $\ep_i( D(\mu, \zeta) )=0$ and $e_i^{(n)} D(s_i\mu, \zeta) = D(\mu, \zeta)$.}
\]
\item If $\langle h_i, \mu \rangle \le 0$ and $s_i \mu \wle \zeta $, then $\ep_i(D(\mu,  \zeta))= - \langle h_i, \mu \rangle$.
\item If $m \seteq -\langle h_i, \zeta\rangle \ge 0$, then
\[
\text{ $\ep_i^*( D(\mu, \zeta) )=0$ and $e_i^{*(m)} D(\mu, s_i\zeta) = D(\mu, \zeta)$.}
\]
\item If $\langle h_i, \zeta \rangle \ge 0$ and $ \mu \wle s_i \zeta $, then $\ep_i^*(D(\mu,  \zeta))=  \langle h_i, \zeta \rangle$.
\end{enumerate}
\end{lem}

\subsection{Quiver Hecke algebras}
Let $\bR$ be a field.
 For $i,j\in I$, we choose polynomials
$\qQ_{i,j}(u,v) \in \bR[u,v]$ such that
\bna

\item $\qQ_{i,j}(u,v) = \qQ_{j,i}(v,u)$,

\item it is of the form
\begin{align*}
\qQ_{i,j}(u,v) =\bc
                   \sum\limits
_{p(\alpha_i , \alpha_i) + q(\alpha_j , \alpha_j)=-2(\alpha_i , \alpha_j)} t_{i,j;p,q} u^pv^q &
\text{if $i \ne j$,}\\[1.5ex]
0 & \text{if $i=j$,}
\ec
\end{align*}
where $t_{i,j;-a_{ij},0} \in  \bR^{\times}$.
\end{enumerate}
We set
\begin{align*}
\bQ_{i,j}(u,v,w)=\dfrac{ \qQ_{i,j}(u,v)- \qQ_{i,j}(w,v)}{u-w}\in \bR[u,v,w].
\end{align*}
For $\beta\in \rlQ_+$ with $ \Ht(\beta)=n$, set
$$
I^\beta\seteq  \left\{\nu=(\nu_1, \ldots, \nu_n ) \in I^n \mid \sum_{k=1}^n\alpha_{\nu_k} = \beta \right\}.
$$
Note that the symmetric group $\mathfrak{S}_n = \langle s_k \mid k=1, \ldots, n-1 \rangle$ acts on $I^\beta$ by place permutations.

\begin{df}
\ For $\beta\in\rlQ_+$,
the {\em quiver Hecke algebra} $R(\beta)$ associated with $\cmA$ and $(\qQ_{i,j}(u,v))_{i,j\in I}$
is the $\bR$-algebra generated by
$$
\{e(\nu) \mid \nu \in I^\beta \}, \; \{x_k \mid 1 \le k \le n \},
 \; \{\tau_l \mid 1 \le l \le n-1 \}
$$
satisfying the following defining relations:
\begin{align*}
& e(\nu) e(\nu') = \delta_{\nu,\nu'} e(\nu),\ \sum_{\nu \in I^{\beta}} e(\nu)=1,\
x_k e(\nu) =  e(\nu) x_k, \  x_k x_l = x_l x_k,\\
& \tau_l e(\nu) = e(s_l(\nu)) \tau_l,\  \tau_k \tau_l = \tau_l \tau_k \text{ if } |k - l| > 1, \\[5pt]
&  \tau_k^2 e(\nu) = \qQ_{\nu_k, \nu_{k+1}}(x_k, x_{k+1}) e(\nu), \\[5pt]
&  (\tau_k x_l - x_{s_k(l)} \tau_k ) e(\nu) = \left\{
                                                           \begin{array}{ll}
                                                             -  e(\nu) & \hbox{if } l=k \text{ and } \nu_k = \nu_{k+1}, \\
                                                               e(\nu) & \hbox{if } l = k+1 \text{ and } \nu_k = \nu_{k+1},  \\
                                                             0 & \hbox{otherwise,}
                                                           \end{array}
                                                         \right. \\[5pt]
&( \tau_{k+1} \tau_{k} \tau_{k+1} - \tau_{k} \tau_{k+1} \tau_{k} )  e(\nu) \\[4pt]
&\qquad \qquad \qquad = \left\{
                                                                                   \begin{array}{ll}
\bQ_{\,\nu_k,\nu_{k+1}}(x_k,x_{k+1},x_{k+2}) e(\nu) & \hbox{if } \nu_k = \nu_{k+2}, \\
0 & \hbox{otherwise}. \end{array}
\right.\\[5pt]
\end{align*}
\end{df}
The algebra $R(\beta)$ has the $\Z$-graded algebra structure defined by
\begin{align*}
\deg(e(\nu))=0, \quad \deg(x_k e(\nu))= ( \alpha_{\nu_k} ,\alpha_{\nu_k}), \quad  \deg(\tau_l e(\nu))= -(\alpha_{\nu_{l}} , \alpha_{\nu_{l+1}}).
\end{align*}

We denote by $R(\beta) \Mod$  the category of graded $R(\beta)$-modules with degree preserving homomorphisms.
We set $R(\beta)\gmod$ to be the full subcategory of $R(\beta)\Mod$ consisting of the modules which are  finite-dimensional over $\bR $, and set
 $R(\beta)\proj$ to be the full subcategory of $R(\beta)\Mod$ consisting of finitely generated  projective graded $R(\beta)$-modules.
We denote $R\Mod \seteq \bigoplus_{\beta \in \rlQ_+} R(\beta)\Mod$, $R\proj \seteq \bigoplus_{\beta \in \rlQ_+} R(\beta)\proj$, and  $R\gmod \seteq \bigoplus_{\beta \in \rlQ_+} R(\beta)\gmod$.
The objects of $R\Mod$ are sometimes called $R$-modules. For simplicity, we write ``a module" instead of ``a graded module".

We denote by $q$ the {\em grading shift functor}, i.e.~$(qM)_k = M_{k-1}$ for a graded module $M = \bigoplus_{k \in \Z} M_k $.
For $M \in R(\beta)\gmod$, the {\em $q$-character} of $M$ is given by
\[
\ch(M) \seteq \sum_{\nu \in I^\beta} \dim_q (e(\nu)M) \nu,
\]
where $ \dim_q V \seteq \sum_{k\in \Z} \dim(V_k)q^{k} $ for a graded vector space $V = \bigoplus_{k\in \Z} V_k $.

For $R(\beta)$-modules $M$ and $N $, $\Hom_{R(\beta)}(M,N)$ denotes the space of degree preserving module homomorphisms.
We set $ \deg(f)=k$ for $f \in \Hom_{R(\beta)}(q^{k}M, N)$, and define
\[
\HOM_{R(\beta)}( M,N ) \seteq \bigoplus_{k \in \Z} \Hom_{R(\beta)}(q^{k}M, N).
\]
We sometimes write $R$ for $R(\beta)$ in $\Hom_{R(\beta)}( M, N )$ and $\HOM_{R(\beta)}( M,N )$ for simplicity.

For $\nu\in I^\beta$ and $\nu'\in I^{\beta'}$, let $e(\nu, \nu')$ be the idempotent corresponding to the concatenation
$\nu\ast\nu'$ of
$\nu$ and $\nu'$, and set
$$
e(\beta, \beta') \seteq \sum_{\nu \in I^\beta, \nu' \in I^{\beta'}} e(\nu, \nu')
{\hs{2ex} \in R(\beta+\beta').}
$$
We define
$$
\Res_{\beta, \beta'} L \seteq e(\beta, \beta')L \qquad \text{ for $L \in R(\beta + \beta')\Mod$,}
$$
and, for $M \in R(\beta)\Mod$ and $N \in R(\beta')\Mod$,
$$
M \conv N \seteq R(\beta+\beta') e(\beta, \beta') \otimes_{R(\beta) \otimes R(\beta')} (M \otimes N).
$$
 Note that $\Res_{\beta, \beta'} L$ is an $R(\beta)\tens R(\beta')$-module and
$M\conv N$ is an $R(\beta+\beta')$-module.
We denote by $M \hconv N$ the head of $M \conv N$ and by $M \sconv N$ the socle of $M \conv N$.
For $i\in I$, the functors $E_i$ and $F_i$ are defined by
\begin{align*}
E_i &: R(\beta+\alpha_i)\Mod \rightarrow R(\beta)\Mod \\
F_i &: R(\beta)\Mod \rightarrow R(\beta+\alpha_i)\Mod
\end{align*}
by
$$
E_i(N) = e(\alpha_i, \beta) N \quad \text{ and }\quad F_i(M) = R(\alpha_i) \conv M
 $$
for $N \in R(\beta+\alpha_i)\Mod $ and $M \in R(\beta)\Mod$.
For $i\in I $ and $n\in \Z_{>0}$, let $L(i)$ be the simple $R(\alpha_i)$-module concentrated on  degree zero and
 $P(i^{(n)})$ the indecomposable  projective $R(n \alpha_i)$-module
whose head is isomorphic to $L(i^n) \seteq q_i^{\frac{n(n+1)}{2}} L(i)^{\conv n}$. Then, for $M\in  R(\beta) \Mod $, we define
\[
F_i^{(n)} M \seteq P(i^{(n)}) \conv M, \qquad
E_i^{(n)} M \seteq \HOM_{R(n\alpha_i)} (P(i^{(n)}), e(n\alpha_i, \beta - n\alpha_i) M).
\]

For $\beta \in \rlQ_+$ and $M \in R(\beta)\Mod$, we define
\begin{align*}
\wt(M) = - \beta
\end{align*}
and, for $i\in I$ and a self-dual simple $R(\beta)$-module $M$,
\begin{align*}
& \ep_i(M) = \max \{ k \ge 0 \mid E_i^k M \ne 0 \}, \quad \ph_i(M) = \ep_i(M) + \langle h_i, \wt(M) \rangle,\\
& \tE_i(M) =  \text{ the self-dual simple module being isomorphic to $\soc( E_i M)$} \\
 & \qquad \qquad \ \text{up to a grading shift,} \\
& \tF_i(M) =  \text{ the self-dual simple module being isomorphic to $L(i) \hconv M$} \\
 & \qquad \qquad \ \text{up to a grading shift.}
\end{align*}
Then the set of the isomorphism classes of
self-dual simple $R$-modules together with the quintuple ($\wt$, $\ep_i$, $\ph_i$, $\tE_i$, $\tF_i$) forms a crystal which is isomorphic to the crystal $B(\infty)$ (\cite{LV09}).
We sometimes ignore the grading shift when using $\tE_i$ and $\tF_i$ if there is no afraid of confusion.
The trivial $R(0)$-module of degree 0 is denoted by $\trivialM$.
We remark that, for a simple $R$-module $M$,
$$
\tE_i^{m} M \simeq E_i^{(m)}M,
$$
when  $m = \ep_i(M)$ (\cite[Lemma 3.8]{KL09}).

We also can define $E_i^*$, $F_i^*$, $\tE_i^*$, $\tF_i^*$, etc.\ in the same manner as above if we replace the roles of $e(\alpha_i, \beta)$ and $R(\alpha_i)\conv -$ with the ones of
$e(\beta, \alpha_i)$ and $- \conv R(\alpha_i)$.

We denote by $K_0(R\proj)$ and $K_0(R\gmod)$ the Grothendieck groups of $R\proj$ and $R\gmod$ respectively.
Then we have the categorification theorem for $U_\Zq^-(\g)$ and $\Aqq[\n]$.

\begin{thm} [{\cite{KL09, KL11, R08}}] \label{Thm: categorification}
There exist isomorphisms of $\Zq$-bialgebras
\[
K_0(R\proj) \simeq  U_\Zq^-(\g) \quad \text{and} \quad K_0(R\gmod) \simeq  \Aqq[\n] .
\]
\end{thm}

Let $\lambda\in\Lambda-\rlQ_+$ and set $ \beta = \Lambda - \lambda$,
$n=\Ht(\beta)$. We define the \emph{cyclotomic quiver Hecke algebra}
\eq
R^{\Lambda}(\lambda)\seteq\dfrac{R(\beta)}{R(\beta)a_{\Lambda}(x_n)R(\beta)},
\label{eq:cyclo}
\eneq
where $a^\Lambda (x_n) = \sum_{\nu \in I^\beta} x_n^{\langle h_{\nu_n}, \Lambda \rangle} e(\nu)$.

We denote by 
 $R^\Lambda(\lambda)\Mod$ the category of graded $R^\Lambda(\lambda)$-modules.
We also denote by $R^\Lambda(\lambda)\proj$ and $R^{\Lambda}(\lambda)\gmod$
 the category of finitely generated  projective graded $R^\Lambda(\lambda)$-modules and
the category of graded $R^{\Lambda}(\lambda)$-modules which are finite-dimensional over $\bR$, respectively.
Their morphisms are homomorphisms homogeneous of degree zero.
We  write  $R^\Lambda\proj \seteq \bigoplus_{\beta \in \rlQ_+} R^\Lambda(\Lambda - \beta)\proj$, $R^\Lambda\gmod \seteq \bigoplus_{\beta \in \rlQ_+} R^\Lambda(\Lambda - \beta)\gmod$, etc.

We define the functors
\begin{align*}
F_i^{\Lambda} &:  R^{\Lambda}(\lambda)\Mod \rightarrow  R^{\Lambda}(\lambda-\alpha_i)\Mod , \\
E_i^{\Lambda} &:  R^{\Lambda}(\lambda)\Mod \rightarrow R^{\Lambda}(\lambda+\alpha_i)\Mod
\end{align*}
by
\begin{align*}
F_i^{\Lambda}M &= R^{\Lambda}(\lambda-\alpha_i)e(\alpha_i,\beta)\otimes_{R^{\Lambda}(\lambda)}M, \\
E_i^{\Lambda}M &= e(\alpha_i,\beta-\alpha_i)M
\end{align*}
for $M\in R^{\Lambda}(\lambda)\Mod$, where $\beta=\Lambda-\lambda\in\rlQ_+$.
 They are exact functors, and
they give a categorification of  $V_q(\Lambda)$.

\begin{thm}[{\cite[Theorem 6.2]{KK11}}]
For $\Lambda \in \wlP^+$, there exist $U_\Zq (\g)$-module isomorphisms
\[
K_0(R^\Lambda\proj) \simeq V_\Zq (\Lambda), \quad K_0(R^\Lambda\gmod) \simeq V_\Zq (\Lambda)^\vee.
\]
\end{thm}

For $n\in\Z_{\ge0}$, the algebra $R(n\al_i)$ acts on $(F_i^\La)^n$ and
$(E_i^\La)^n$ and we define
\eqn
F_i^{\La\,(n)}\,M&=&\HOM_{R(n\al_i)}\bl P(i^{ (n) }),(F_i^\La)^n\,M\br,\\
E_i^{\La\,(n)}\,M&=&\HOM_{R(n\al_i)}\bl P(i^{ (n) }),(E_i^\La)^n\,M\br.
\eneqn
Then we obtain exact functors:
\eq
&&\ba{rcl}
F_i^{\Lambda\,(n)} &:&  R^{\Lambda}(\lambda)\Mod \rightarrow  R^{\Lambda}(\lambda-n\alpha_i)\Mod , \\
E_i^{\Lambda\,(n)} &:&  R^{\Lambda}(\lambda)\Mod \rightarrow R^{\Lambda}(\lambda+n\alpha_i)\Mod.
\ea\eneq
Then the following lemma is an easy consequence of the theory of
$\mathfrak{sl}_2$-categorification due to Rouquier (\cite{R08}).
\Lemma\label{lem:htlt}
Let $\Lambda \in \wlP^+$ and $\la\in\weyl\La$ such that $n\seteq\ang{h_i,\la}\ge0$.
Then we have category equivalences, quasi-inverse to each other:
$$\xymatrix@C=9ex{
R^{\Lambda}(\lambda)\Mod\ar@<.7ex>[r]^{F_i^{\Lambda\,(n)}}&
R^{\Lambda}(s_i\lambda)\Mod\,.\ar@<.7ex>[l]^{E_i^{\Lambda\,(n)}}
}$$
\enlemma

\subsection{Convex preorders}
In this subsection, we review the notion of convex preorders introduced in \cite{TW16}, and show  some of their properties.

\begin{df}
\bnum
\item A \emph{preorder} $\preceq$ on a set $X$ is a binary relation on $X$ satisfying
\bna
\item $x \preceq x$ for any $x\in X$,
\item if $x \preceq y$ and $y \preceq z$ for $x,y,z \in X$, then $x \preceq z$.
\end{enumerate}
\item A preorder $\preceq$ on a set $X$ is \emph{total} if, for any pair $(x,y)$ of elements of $X$, we have either $x \preceq y$ or $y \preceq x$.
\item For a preorder $\preceq$, we say that $x$ and $y$ are \emph{$\preceq$-equivalent} if $x \preceq y$ and $y \preceq x$.
An equivalent class for $\preceq$ is called $\preceq$-equivalence class.
\item We write $x \prec y$ \ro resp.\ $x \succ y$\rf\ if $x \preceq y$ and $ x \not\succeq y$ \ro resp.\  $x \succeq y$ and $ x \not\preceq y$\rf.
\item For subsets $A$ and $B$, we write $A \preceq B$ if $a \preceq b$ for any $a \in A$ and $b \in B$.
We also define $A \prec B$, $A \succ B$, etc., in a similar manner.
\end{enumerate}
\end{df}

For preorders $\preceq$ and $\preceq'$ on $X$, we say $\preceq'$ is a \emph{refinement} of $\preceq$ if $x \prec y$ implies $x \prec' y$.

\begin{df} \label{Def: face}
A \emph{face} is a decomposition of a subset $X$ of an $\R$-vector space  into three disjoint subsets $ X = A_- \sqcup A_0 \sqcup A_+ $ such that
\begin{align*}
( \Sp  A_+ + \SpR A_0) \cap \Sp A_- &= \{ 0\}, \\
( \Sp  A_- + \SpR A_0) \cap \Sp A_+ &= \{ 0\},
\end{align*}
where $\SpR S $ is the $\R$-vector space spanned by $S$ and
$\Sp S $ is the subset of $\SpR S $ whose elements are linear combinations of $S$ with non-negative coefficients.
\end{df}
Note that $\SpR \emptyset = \Sp \emptyset = \{ 0 \}$  by the definition.
One can prove the following lemma easily.
\begin{lem} \label{Lem: face}
Let $(A_-, A_0, A_+)$ be a decomposition of a subset $X$ in an $\R$-vector space. Then the following are equivalent:
\bna
\item the triple $(A_-, A_0, A_+)$ is a face,
\item it satisfies
\begin{align*}
&\Sp(A_+ \cup A_0) \cap \Sp(A_- \cup A_0) = \Sp A_0, \\
&\Sp A_+ \cap \SpR A_0 = \Sp A_- \cap \SpR A_0 = \{0\},
\end{align*}
\item for any $x_\pm \in \Sp A_\pm$, if $x_- - x_+ \in \SpR A_0$, then $x_- = x_+ = 0.$
\end{enumerate}
\end{lem}

Note that Lemma \ref{Lem: face} tells that our definition of a face of $\prD$ is equivalent to that given in \cite{MT16}.

\begin{rem}
Let $(A_-, A_0, A_+)$ be a decomposition of a subset $X$ in an $\R$-vector space.
If there exists a linear functional $f$ such that
\[
A_- \setminus \{0\} \subset f^{-1}(\R_{<0}), \quad
A_0  \subset f^{-1}(0), \quad
A_+ \setminus \{0\} \subset f^{-1}(\R_{>0}),
\]
then the triple $(A_-, A_0, A_+)$ is a face. Conversely, if
 $(A_-, A_0, A_+)$ is a face and
$A_-$ and $A_+$ are finite subsets,
then there exists such a linear functional  $f$.
\end{rem}

\begin{df} \label{Def: convex}
Let $V$ be an $\R$-vector space and let $X$ be a subset of $V \setminus \{0\}$.
\bnum
\item A \emph{convex} preorder $\preceq$ on $X$  is a total preorder on $X$ such that, for any $\preceq$-equivalence class $\eC$,
the triple
$(  \{ x \in X \mid x \prec \eC \}, \eC,  \{ x \in X \mid x \succ \eC \} )$
is a face.
\item A convex preorder $\preceq$ on $X$ is called a \emph{convex order} if every $\preceq$-equivalence class is of the form $X \cap l$
for some line $l$ in $V$ through the origin.
\end{enumerate}
\end{df}
Note that a convex order is not an order on $X$. It is an order on the image of $X$ by the projection $V \setminus \{0\} \rightarrow (V \setminus \{0\})/\R^\times$.

The lemma below follows immediately from Definition \ref{Def: convex} and Lemma \ref{Lem: face}.

\begin{lem} \label{Lem: convex}
Let $V$ be an $\R$-vector space and let $\preceq$ be a convex preorder on $X \subset V \setminus \{0\}$.
\bnum
\item If $\beta, t \beta \in X$ for some $t \in \R$, then $\beta$ and $t\beta$ are in the same $\preceq$-equivalence class.
\item If $\alpha, \beta, \gamma \in X$ with $\alpha+ \beta = \gamma$ and $\alpha \prec \gamma$, then $\gamma \prec \beta$.
\item If $\alpha, \beta, \gamma \in X$ with $\alpha+ \beta = \gamma$ and $\gamma \prec \beta$, then $\alpha \prec \gamma$.
\end{enumerate}
\end{lem}

\begin{lem} \label{Lem: convex2}
Let $V$ be an $\R$-vector space and let $\preceq$ be a convex preorder on $X \subset V \setminus \{0\}$
such that $\Sp X \cap ( - \Sp X ) = \{ 0\}.$
Let $\eC$ be a $\preceq$-equivalence class and set
\[
\eC_- = \{ x\in X \mid x \prec \eC \} \quad \text{ and }\quad  \eC_+ = \{ x\in X \mid x \succ \eC \}.
\]
 Let $\beta$ be an element of $\Sp \eC'$
for some $\preceq$-equivalence class $\eC'$.
\bnum
\item
If $\beta$ can be written as $\beta = \beta_1 + \beta_2$ for some $\beta_1 \in \Sp (\eC_- \cup \eC)$ and some non-zero $\beta_2 \in \Sp \eC_-$,
then $ \beta \in \Sp \eC_-$.
\item
If $\beta$ can be written as $\beta = \beta_1 + \beta_2$ for some $\beta_1 \in \Sp (\eC_+ \cup \eC)$ and some non-zero $\beta_2 \in \Sp \eC_+$,
then $ \beta \in \Sp \eC_+$.
\end{enumerate}
\end{lem}
\begin{proof}
(i)\
The element
$\beta$ is contained in one of
$\Sp \eC_-$, $\Sp\eC$ or $\Sp\eC_+$.
If $\beta\in \Sp\eC_+$, then
$\beta\in\Sp(\eC_+)\cap \bl\Sp(\eC)+\Sp(\eC_-)\br=\{0\}$.
Suppose that $\beta  \in \Sp\eC$. If we write $\beta_1 = \beta_1^- + \beta_1^0$
for some $\beta_1^- \in \Sp \eC_-$ and some $\beta_1^0 \in \Sp \eC$, then we have
\[
\beta = \beta_1 + \beta_2 = \beta_1^- + \beta_1^0 + \beta_2.
\]
This tells that $ \beta_1^- + \beta_2 = \beta -  \beta_1^0 \in \Sp \eC_- \cap \SpR \eC = \{ 0\}$.
As $\Sp X \cap (- \Sp X) = \{ 0\}$,  we have $ \beta_2 = - \beta_1^- = 0 $, which is a contradiction to the fact that $\beta_2 \ne 0$.

\smallskip
\noi
(ii) can be proved in the same manner as above.
\end{proof}

\begin{lem} \label{Lem: convex preorder on H}
Let $n,m$ be positive integers such that $1<m \le n$, and $V =\R^n$. We define
\eqn
H&=&\bigcup_{k=1}^m  \{ x=(x_1, \ldots, x_n)\in V \mid \text{$x_j=0$ for $1\le j<k$ and $x_k>0$} \}
\subset V \setminus \{0\},
\eneqn
and set $H_+= \{x\in H \mid x_1>0 \}$ and
$H_0=\{x\in H \mid x_1=0 \}$.
We define the total preorder $\preceq$ on $H$ as follows:
\bna
\item $H_0$ is a $\preceq$-equivalence class, and
$H_+\prec H_0$,
\item for $x,y\in H_+$, we define
$$x\prec y\Longleftrightarrow
\parbox[t]{40ex}{there exists $k$ such that
$2\le k\le m$ and $\dfrac{x_j}{x_1}=\dfrac{y_j}{y_1}$ for $2\le j<k$
and  $\dfrac{x_k}{x_1}<\dfrac{y_k}{y_1}$.}$$
\end{enumerate}
Then $\preceq$ is a convex preorder on $H$.
Moreover $\dim \SpR(\eC)<\dim V$ for any $\preceq$ equivalence class $\eC$.
\end{lem}
\begin{proof}
It is obvious that $\preceq$ is a total preorder on $H$.
Let $\eC$ be a $\preceq$-equivalence class.
We set
\[
\eC_-=\{x\in H \mid x\prec \eC \} \quad \text{and} \quad \eC_+=\{x\in H \mid x\succ \eC \}.
\]
Let us show that
\begin{align} \label{Eq: shc}
\text{if $x_\pm \in \Sp(\eC_\pm)$ satisfy
$x_- -x_+\in \SpR(\eC)$, then $x_-=x_+=0$.}
\end{align}
Note that any equivalence class $\eC$ is either $H_0$ or
$$
\eC(\lambda)
\seteq \{x\in H \mid \text{$x_k=\lambda_k x_1$ for $2\le k\le m$} \}
$$
for $\lambda=(\lambda_2,\ldots,\lambda_m)\in\R^{m-1}$.

If $\eC=H_0$, then
$\eC_-=H_+=\Sp(\eC_-) \setminus \{0\}$, and $\eC_+=\emptyset$.
Hence $\eqref{Eq: shc}$ holds.

Suppose that $\eC=\eC(\lambda)$ for $\lambda=(\lambda_2,\ldots,\lambda_m)\in\R^{m-1}$.
Then we have
\begin{align*}
\eC_-&=
\left\{x\in H_+ \mid \parbox{38ex}{there exists $k$ such that
$2\le k\le m$,  $x_k<\lambda_k x_1$ and $x_j=\lambda_j x_1$ for $2\le j<k$} \right\}\\
&= \left\{x\in H \mid \parbox{38ex}{there exists $k$ such that
$2\le k\le m$,  $x_k<\lambda_k x_1$ and $x_j=\lambda_j x_1$ for $2\le j<k$} \right\}\\
&=\Sp(\eC_-)\setminus \{0\}, \\
\eC_+&=H_0\cup
\left\{x\in H_+ \mid \parbox{38ex}{there exists $k$ such that
$2\le k\le m$,  $x_k>\lambda_k x_1$ and $x_j=\lambda_j x_1$ for $2\le j<k$} \right\}\\
&= \left\{x\in H \mid \parbox{38ex}{there exists $k$ such that
$2\le k\le m$,  $x_k>\lambda_k x_1$ and $x_j=\lambda_j x_1$ for $2\le j<k$} \right\}\\
&=\Sp(\eC_+)\setminus \{0\},\\
\SpR(\eC)&=\{x\in H \mid \text{$x_k=\lambda_kx_1$ for $2\le k\le m$} \}.
\end{align*}
Then one can easily show that \eqref{Eq: shc} holds.
\end{proof}

\begin{lem} \label{Lem: refinement}
Let $V$ be an $\R$-vector space, and
 $\preceq$ a convex preorder on
$ X \subset  V \setminus \{0\}$.
For each $\preceq$-equivalence class $\eC$
let $\preceq_{\eC}$ be an arbitrary convex preorder on $\eC$.
Let $\preceq'$ be the binary relation on $X$ defined as follows:
$$\parbox{75ex}{$x\preceq' y$ if
either $x\prec y$ or \\\hs{15ex}
$x$ and $y$ belong to the same $\preceq$-equivalence class $\eC$ and
$x\preceq_\eC y$.}$$
Then $\preceq'$ is a convex preorder on $X$.
\end{lem}
\begin{proof}
It is obvious that $\preceq'$ is a total preorder on $X$.
We shall show that $\preceq'$ is a convex preorder.
Let $\eC_0$ be a $\preceq'$-equivalence class.
Then $\eC_0$ is contained in some $\preceq$-equivalence class $\eC$.
Then $A_- \seteq \{x\in X \mid x\prec' \eC_0 \}$
is equal to
$B_-\cup C_-$ where
$B_-=\{x\in X \mid x\prec\eC\}$ and $C_-=\{x\in \eC \mid x\prec_\eC \eC_0\}$.
Similarly, we have $A_+=B_+\cup C_+$ where
$A_+=\{x\in X \mid x\succ' \eC_0\}$,
$B_+=\{x\in X \mid x\succ \eC\}$ and $C_+=\{x\in \eC \mid x\succ_\eC \eC_0\}$.

We take $x_-\in\Sp(A_-)$ and $x_+\in\Sp(A_+)$ such that $x_- -x_+\in\SpR(\eC_0)$.
Then we can write
$$x_\pm=y_\pm+z_\pm$$
for
some $y_\pm\in \Sp(B_\pm)$ and some $z_\pm\in\Sp(C_\pm)$.
Since $x_- -x_+$ and $z_+-z_- $ are contained in $\SpR(\eC)$,
we have
$$
y_--y_+ = x_- - x_+ + z_+-z_-\subset\SpR(\eC),$$
which implies $y_-=y_+=0$.
Hence $z_--z_+\in\SpR(\eC_0)$, which implies $z_-=z_+=0$.
\end{proof}

\begin{prop} \label{Prop: extension}
Let $V$ be a finite-dimensional $\R$-vector space, and let $X$ be a subset of $V \setminus \{0\}$ such that
$$
\Sp X \cap ( - \Sp X ) = \{ 0\}.
$$
For any convex preorder $\preceq$ on $X$,
there exists a convex order on $X$ which refines $\preceq$.
\end{prop}
\begin{proof}
Let $\preceq$ be a convex preorder on $X$. We set $n \seteq \dim V$.

We shall use induction on $n$. We first assume that $X$ itself is only one $\preceq$-equivalence class.
Since the assertion is obvious if $n=0$, we assume
$n \ge 1$.
Since $\Sp(X)\not=V$,
there exists a non-zero linear functional $f_1$ such that
$X\subset f_1^{-1}(\R_{\ge0})$.
If $\dim V\ge2$, there exists a linear functional $f_2$ such that
$X\cap f_1^{-1}(0)\subset f_2^{-1}(\R_{\ge0})$ and $f_1$ and $f_2$ are linearly independent.
Continuing this process to $n$, we can find an isomorphism
$f : V \buildrel \sim \over \rightarrow \R^n$ such that $X\subset f^{-1}(H)$, where
$H$ is given in Lemma~\ref{Lem: convex preorder on H}.
Let $\preceq$ be the convex preorder on $X$ induced by the convex preorder on $H$.
Since $\dim\SpR(\eC)< n$ for any $\preceq$-equivalence class $\eC$,
the induction hypothesis implies that
there exists a convex order $\preceq_\eC$ on each $\eC$.
Thanks to Lemma \ref{Lem: refinement},
we obtain a desired convex order on $X$ by refining $\preceq$ by the $\preceq_\eC$'s.

We now assume that $X$ is not a $\preceq$-equivalence class. Since
$\dim \SpR \eC < n$ for any $\preceq$-equivalence class $\eC$,
every $\preceq$-equivalence class $\eC$ has a convex order $\preceq_{\eC}$
by the induction hypothesis.
Using Lemma \ref{Lem: refinement}, we obtain the assertion by refining $\preceq$ by $\preceq_\eC$'s.
\end{proof}

\begin{df} \label{Def: charge}
  Let $V$ be a finite-dimensional $\R$-vector space.
A \emph{charge} on $X \subset V \setminus \{0\} $ is an $\R$-linear function $c: V \rightarrow \C$ such that
\bna
\item $c(X)$ does not contain $0$,
\item $\Sp c(X) \cap ( - \Sp c(X)  ) = \{0\}$.
\end{enumerate}
\end{df}

As $\Sp c(X) \setminus \{0\}$ is simply connected, we can define $ \arg: c(X) \rightarrow \R$, which is unique up to a constant function.
We set
\[
\arg_c \seteq \arg \circ c: X \longrightarrow \R.
\]

\begin{lem} \label{Lem: charge}
Let  $c: V \rightarrow \C$ be a charge on $X \subset V \setminus \{0\}$.  Then the binary relation $ \preceq_c$ on $X$ defined by
\[
x \preceq_c y \quad \text{ if and only if } \quad \arg_c x \le \arg_c y
\]
is a convex preorder on $X$.
\end{lem}
\begin{proof}
It is obvious that $\preceq_c$ is a preorder on $X$. 
Without loss of generality, we may assume that $0 \le \arg_c x \le \pi$ for any $x \in X$ by Definition \ref{Def: charge}.
We take a $\preceq$-equivalence class $\eC$. Let $t \seteq \arg_c x$ for $x \in \eC$, and
\begin{align*}
\eC_- &:= \{ y \in X \mid y \prec_c \eC \} = \{ y \in X \mid 0 \le \arg_c y < t \}, \\
\eC_+ &:= \{ z \in X \mid z \succ_c \eC \} = \{ z \in X \mid  t < \arg_c z \le \pi \}.
\end{align*}

Since $\arg_c y < t $ for $y\in \Sp \eC_-  \setminus \{0\}$ and
 $ \arg_c x = t$
for $x \in \SpR\eC \setminus \{0\}$,
we have $\arg_c (y + x) < t$.
This implies
\[
( \Sp  \eC_- + \SpR \eC) \cap \Sp \eC_+ = \{ 0\}.
\]

In the same manner, one can prove
\[
( \Sp  \eC_+ + \SpR \eC) \cap \Sp \eC_- = \{ 0\}.
\]
Therefore the triple $(\eC_-, \eC, \eC_+)$ is a face, which completes the proof.
\end{proof}

The following proposition was informed to us by  Peter McNamara.

\begin{prop} \label{Prop: convex preorder for w}
Let $\underline{w} = s_{i_1}s_{i_2}\cdots s_{i_l}$ be a reduced expression of $w \in \weyl$.
We set $\beta_k \seteq s_{i_1}\cdots s_{i_{k-1}}(\alpha_{i_k}) $ so that $\prD \cap w\nrD = \{ \beta_1, \ldots, \beta_l \}$.
Then there is a convex preorder $\preceq$ on $\prD$ such that
$$
\beta_1 \prec \beta_2 \prec \cdots \prec \beta_l \prec \gamma
$$
for any $\gamma \in \prD \cap w\prD$.
\end{prop}
\begin{proof}
It is easy to see that $ c \seteq \rho^\vee + \sqrt{-1} \rho^\vee \circ w^{-1} $ is a charge on $\prD$.
Here $\rho^\vee \in \wlP^\vee$ is defined by $\langle\rho^\vee, \alpha_i \rangle = 1 $ for $i\in I$. We define the total preorder on $\prD$ by
\begin{align*}
& \beta_1 \prec \cdots \prec \beta_l \prec \beta \quad \text{ for any } \beta \in \prD \cap w\prD,  \\
& \beta \preceq \beta' \Longleftrightarrow \arg_c \beta \le \arg_c \beta' \quad \text{ for } \beta, \beta' \in \prD \cap w\prD.
\end{align*}
We now show that $\preceq$ is a convex preorder on $\prD$.

For $k=0,1, \ldots, l$, we set $w_k \seteq s_{i_1} \cdots s_{i_k}$ and $ \phi_k \seteq \rho^\vee \circ w_k^{-1} $.
Since
\[
\text{$ \beta_j \in w_k \nrD$ for $ 1 \le j \le k$, \quad $ \beta_j \in w_k \prD$ for $ k < j \le l$ , \quad  $\prD \cap w\prD \subset w_k \prD$,}
\]
we have
\begin{align*}
& \{ \beta_1, \ldots, \beta_k  \} \subset \phi_k^{-1}(\R_{<0}), \\
& \{ \beta_{k+1}, \ldots, \beta_l  \} \cup (\prD \cap w\prD) \subset \phi_k^{-1}(\R_{>0}).
\end{align*}
Letting
\[
\eC_- \seteq \{ \beta \in \prD \mid \beta \prec \beta_k  \}, \quad \eC\seteq\{\beta_k\}, \quad \eC_+ \seteq \{  \beta \in \prD  \mid \beta \succ \beta_k \},
\]
the hyperplane  $ \phi_{k-1}^{-1}(0)$ (resp.\ $ \phi_{k}^{-1}(0)$)
divides $\prD$ into $\eC_-$ and $\eC \cup \eC_+$ (resp.\ $\eC_-\cup \eC$ and $ \eC_+$).
Thus, the triple $( \eC_- , \eC, \eC_+  )$ satisfies the conditions in Lemma \ref{Lem: face}, which tells the triple is a face.

We now take $\beta \in \prD \cap w\prD$ and let $\eC$ denote the $\preceq$-equivalence class containing $\beta$. Let $ t = \rho^\vee(w^{-1}\beta)/ \rho^\vee(\beta )$ and
\[
\eC_- \seteq \{ \gamma \in \prD \mid \gamma \prec \eC  \},  \quad \eC_+ \seteq \{  \gamma \in \prD  \mid \gamma \succ \eC\}.
\]
By the definition of $c$, we have
\begin{align*}
\eC_- &=  (\prD \cap w \nrD) \cup  \{  \gamma \in \prD \cap w\prD \mid \rho^\vee(w^{-1} \gamma) < t \rho^\vee(\gamma)   \} \\
 &= \{  \gamma \in \prD  \mid \rho^\vee(w^{-1} \gamma) < t \rho^\vee(\gamma) \}
 = \{ \gamma \in \prD \mid \gamma \prec_c \beta \}, \\
\eC &=  \{  \gamma \in \prD  \mid \rho^\vee(w^{-1} \gamma) = t \rho^\vee(\gamma) \} =
 \{ \gamma \in \prD \mid \gamma \preceq_c \beta, \ \gamma \succeq_c \beta \}, \\
\eC_+ &=  \{  \gamma \in \prD  \mid \rho^\vee(w^{-1} \gamma) > t \rho^\vee(\gamma) \}=
\{ \gamma \in \prD \mid \gamma \succ_c \beta \}.
\end{align*}
By Lemma \ref{Lem: charge}, the triple $( \eC_- , \eC, \eC_+  )$ is a face. Therefore, $\preceq$ is a convex preorder on $\prD$.
\end{proof}

We sometimes denote by $\preceq^{\underline{w}}$ a convex order which refines the convex preorder corresponding to $\underline{w}$ given in Proposition \ref{Prop: convex preorder for w}
if we need to emphasize the reduced expression $\underline{w}$.

\vskip 2em

\section{Categorification of Doubly-invariant algebras} \label{Sec: Cwv}

\subsection{Cuspidal decomposition}

In order to recall the notion of cuspidal decomposition,
we introduce the following notations.

\begin{df}
For $M \in  R(\beta)\Mod$, we define
\begin{align*}
\gW(M) &:= \{  \gamma \in  \rlQ_+ \cap (\beta - \rlQ_+)  \mid  e(\gamma, \beta-\gamma) M \ne 0  \}, \\
\sgW(M) &:= \{  \gamma \in  \rlQ_+ \cap (\beta - \rlQ_+)  \mid  e(\beta-\gamma, \gamma) M \ne 0  \}.
\end{align*}
\end{df}

We have $\sgW(M) = \beta - \gW(M)$ by the definition. Note that $0,\beta \in \gW(M)$ if $M \ne 0$.
From the definition, we have the following lemma.

\begin{lem} \label{Lem: sum of W}
Let $M$ and $N$ be $R$-modules. Then
\[
\gW(M\conv N) = \gW(M) + \gW(N),
\qquad
\sgW(M\conv N) = \sgW(M) + \sgW(N).
\]
\end{lem}

The following proposition is a consequence of \cite[Proposition 3.7]{TW16}.

\begin{prop} 
 \label{Prop: Span W(L)}
Let $L$ be a simple $R$-module. Then
\begin{align*}
\Sp \gW(L) = \Sp ( \gW(L) \cap \prD ), \quad \Sp \sgW(L) = \Sp ( \sgW(L) \cap \prD ).
\end{align*}
\end{prop}

\bigskip
Let us fix a convex {\em order} $\preceq$ on $\prD$ in this subsection.
We set
\eqn
\prD^{\min}&\seteq&\set{\beta\in\prD}{\Q\beta\cap\prD  \subset  \Z_{>0}\beta},\\
\Z_{>0} \prD &\seteq& \{ n \gamma \mid n \in \Z_{>0},\ \gamma \in \prD \}\\
&=&\{ n \gamma \mid n \in \Z_{>0},\ \gamma \in \prD^{\min} \}.
\eneqn
Note that $\beta \in \Z_{>0} \prD $ if and only if $\beta \in\rlQ_+\cap
 \SpR  \eC $ for some $\preceq$-equivalence  class $\eC$ of $\prD$.

 We can extend the convex order $\preceq$ on $\prD$
uniquely to the convex order on $\Z_{>0}\prD$ as follows:
For $\beta, \beta' \in \Z_{>0}\prD$, there are $\preceq$-equivalence classes $\eC$ and $\eC'$ such that
$\beta \in \Sp \eC$ and $\beta' \in \Sp\eC'$. Then
\begin{align} \label{Eq: preceq}
\text{we write $\beta \preceq \beta'$ if $\eC \preceq \eC'$.}
\end{align}
Note that
\begin{equation} \label{Eq: preceq and prD}
\begin{aligned}
 &\parbox{70ex}{for any $\preceq$-equivalence class $\eC$ of $\Z_{>0}\prD$, there is a unique $\alpha \in \prD^{\min}$ such that
$ \eC = \Z_{>0} \alpha $.}
\end{aligned}
\end{equation}

\begin{df} \label{Def: cuspidal}
Let $\beta \in \rlQ_+ \setminus \{0\}$. A simple $R(\beta)$-module $L$ is \emph{$\preceq$-cuspidal} if
\bna
\item $\beta \in \Z_{>0} \prD $,
\item$\gW(L) \subset \Sp  \{ \gamma \in \prD \mid \gamma \preceq \beta \} $.
\end{enumerate}
\end{df}
\begin{rem}
The cuspidal modules in our definition are called semi-cuspidal in \cite{TW16}.
\end{rem}

 By Lemma \ref{Lem: convex}, a simple $R(\beta)$-module $L$
is cuspidal if and only if
\begin{align*}
\beta  \in \Z_{>0}\Delta_+ \ \text{and} \    \sgW(L) \subset \Sp  \{ \gamma \in \prD \mid \gamma \succeq \beta\} .
\end{align*}

\begin{prop}[{\cite[Corollary 2.17]{TW16}}]  \label{Prop: n cuspidal}
Let $\beta \in \Z_{> 0}\prD$.
Then,
the number of the isomorphism classes of
$\preceq$-cuspidal simple $R(\beta)$-modules is
\[
\sum_{\substack{\beta=\gamma_1+\cdots+\gamma_n,\\[.5ex]
\gamma_k\in \Q\beta\cap \prD }}
\prod_{k=1}^n m_{\gamma_k},
\]
where the sum ranges over all ways of writing
$\beta$ as a sum of positive roots in $\Q\beta\cap\prD$  \ro up to permutations\rf.
Here, $m_\gamma=\dim \g_\gamma$ for $\gamma\in\prD$.
\end{prop}

\begin{prop}[{\cite[Proposition~2.21]{TW16}}]\label{prop:realcuspidal}
Let $\beta$ be a real positive root.
\bnum
\item For $n \in \Z_{>0}$, there exists a unique self-dual $\preceq$-cuspidal $R(n\beta)$-module $L(n\beta)$ up to an  isomorphism.
\item For $n \in \Z_{>0}$, $L(\beta)^{\conv n}$ is simple and isomorphic to $L(n\beta)$ up to a grading shift.
\end{enumerate}
\end{prop}

\begin{thm}[{\cite[Theorem 2.19]{TW16}}]\label{Thm: cuspidal decomposition}
For  a simple $R(\beta)$-module $L$, there exists a unique sequence $(L_1,L_2, \ldots, L_h)$ of $\preceq$-cuspidal modules \ro up to isomorphisms\rf\ such that
\bna
\item $-\wt(L_k) \succ -\wt(L_{k+1}) $ for $k=1, \ldots, h-1$,
\item $L$ is isomorphic to the head of $L_1 \conv L_2 \conv \cdots \conv L_h$.
\end{enumerate}
\end{thm}

The sequence $(L_1, L_2, \ldots, L_h )$ associated to $L$ given in Theorem \ref{Thm: cuspidal decomposition}
is called the \emph{$\preceq$-cuspidal decomposition} of $L$, which is  denoted by
\[
\cd_{\preceq}(L) \seteq (L_1, \ldots, L_h).
\]
We write $\cd$ instead of $\cd_{\preceq}$ for simplicity when there is no afraid of confusion.
Note that we have
\eqn
\gW(L) \subset  \Sp  \{ \gamma \in \prD \mid \gamma \preceq -\wt(L_1)\},\\
\sgW(L) \subset  \Sp  \{ \gamma \in \prD \mid \gamma \succeq -\wt(L_h)\}.
\eneqn

\begin{df} Let $M_k$ be $R(\beta_k)$-modules for $k=1, \ldots, n$. We say that  $ \allowbreak( M_1,  \ldots, M_n)$ is \emph{unmixed} if
\[
\Res_{\beta_1, \beta_2, \ldots, \beta_n} (M_1 \conv M_2 \conv \cdots \conv M_n) = M_1 \otimes M_2 \otimes \cdots \otimes M_n.
\]
\end{df}

Then we have the following proposition.
\begin{prop} [\protect{cf.\ \cite[Proposition 2.6, Lemma 2.8]{TW16}}]
Let $M_k$ be $R(\beta_k)$-modules for $k=1, \ldots, n$. Suppose that $(M_1, \ldots, M_n)$ is unmixed.
\bnum
\item For $w \in \sg_n$,
\begin{equation*}
\begin{aligned}
  \Res_{\beta_{1}, \beta_{2}, \ldots, \beta_{n}} (M_{w(1)} \conv M_{w(2)}\conv \cdots \conv M_{w(n)})
\simeq q^{t_w} M_1 \otimes M_2\otimes \cdots \otimes M_n ,
\end{aligned}
\end{equation*}
where $t_w = - \sum_{ i< j \text{ and } w^{ -1 }(i)> w^{ -1 }(j) } (\beta_i, \beta_j) $.
\item Assume that $M_j$ is  simple for $1 \le j \le n$. Then the head $M$ of $M_1 \conv M_2 \conv  \cdots \conv M_n$ is simple. Moreover, we have
\begin{equation*}
\begin{aligned}
\Res_{\beta_1, \beta_2, \ldots, \beta_n} (M) \simeq M_1 \otimes M_2\otimes \cdots \otimes M_n.
\end{aligned}
\end{equation*}
\end{enumerate}
\end{prop}

\begin{lem}[{\cite[Lemma 2.9]{TW16}}] \label{Lem: unmixing}
Let $(L_1,\ldots,L_h)$ be a sequence of cuspidal simple $R$-modules with $-\wt(L_1) \succ \cdots \succ -\wt(L_h)$. Then $(L_1,\ldots,L_h)$ is unmixed.
In particular, for a simple $R$-module $L$, the $\preceq$-cuspidal decomposition $\cd(L)$ is unmixed.
\end{lem}

For $m,n \in \Z_{\ge0}$, we denote by $w[m,n]$ the element of $\sg_{m+n}$ defined by
$$
w[m,n](k) = \bc
                k+n & \text{if $1 \le k \le m$,} \\
                k-m & \text{if $m < k \le m+n$.}
\ec
$$

The following proposition generalizes \cite[Proposition 10.1.3]{KKKO17}.
\begin{prop}  [cf.\ {\cite[Lemma 2.8]{TW16}}]  \label{Prop: R-matrix for ep=0}
Let $M \in R(\beta)\gmod$ and $N \in R(\gamma)\gmod$, and set $m = \Ht(\beta)$ and $n = \Ht(\gamma)$.
Suppose that $\sgW(M) \cap \gW(N) \subset\{0\}$. Then we have
\bnum
\item  $\Res_{\beta, \gamma} (M\conv N) = M\otimes N $ and $\Res_{\beta, \gamma} (N\conv M) \simeq q^{-(\beta, \gamma)} M\otimes N $,
\item the map
\[
u \otimes v \mapsto \tau_{w[n,m]}(v \otimes u) \qquad \text{ for } u\in M, v\in N,
\]
 induces an $R(\beta+\gamma)$-linear homomorphism $M \conv N \rightarrow q^{(\beta, \gamma)}N \conv M$.
\end{enumerate}
\end{prop}
\begin{proof}
For $\beta_1 \in \gW(M)$ and $\gamma_1 \in \gW(N)$ such that $\beta = \beta_1 + \gamma_1$, we have
\[
\beta - \beta_1 = \gamma_1 \in \sgW(M) \cap \gW(N).
\]
 Hence we have
$\beta_1=\beta$ and $\gamma_1=0$.
 This implies that
\[
e(\beta, \gamma) (M\conv N) = M \otimes N, \quad e(\beta, \gamma) (N\conv M) = \tau_{w[n,m]} (N \otimes M),
\]
which gives (i).
Moreover, the map
\[
u \otimes v \mapsto \tau_{w[n,m]}(v \otimes u) \qquad \text{ for } u\in M, v\in N,
\]
is an $R(\beta)\otimes R(\gamma)$-module homomorphism $M \otimes N \rightarrow q^{(\beta, \gamma)}e(\beta, \gamma ) (N \conv M)$.
Indeed, for example, for $1\le k\le m$,
\eqn
\bl x_k\tau_{w[n,m]}-\tau_{w[n,m]}x_{k+n}\br (v\tens u)
&  \in  &e(\beta,\gamma)(N\conv M)\cap
\bl\sum_{w<w[n,m]}\tau_w(N\tens M)\br=\{0\},
\eneqn
where the last equality follows from (i).
\end{proof}

\begin{lem} \label{Lem: unmixed sep}
Let $M_{t }$ be an $R(\beta_{ t })$-module for $1 \le t \le n$ and let $2 \le k \le n$.
If
$(M_1,\ldots,M_{k-1})$,  $(M_k,\ldots,M_n)$ and $(M_1\conv \cdots \conv M_{k-1}, M_k\conv \cdots \conv M_n)$ are unmixed, then $(M_1,\ldots,M_n)$ is unmixed.
\end{lem}
\begin{proof}
It follows from the transitivity of the induction and  restriction.
\end{proof}

\begin{lem} \label{Lem: unmixed conv}
Let $L$ be a simple $R(\beta)$-module with $\cd(L) = (L_1,\ldots,L_h)$.
Then $$(L_1,\cdots,L_{k-1}, (L_k \conv \cdots \conv L_h))$$ is unmixed for any $2\le k\le h$.
\end{lem}
\begin{proof}
Note that
\eqn
&&\gW^*(L_1 \conv \cdots\conv L_{k-1}) \cap \gW(L_k \conv \cdots\conv L_h)  \\
&&\hskip2em \subset  \Sp \set{\gamma \in \Delta_+}{ \gamma \succeq -\wt(L_{k-1}) } \cap \Sp \set{\gamma \in \Delta_+}{ \gamma \preceq -\wt(L_{k})} \subset \{0\}.\eneqn
By Proposition \ref{Prop: R-matrix for ep=0}, we conclude that $((L_1 \conv \cdots \conv L_{k-1}), (L_k \conv \cdots \conv L_h))$ is unmixed.
Since $(L_1,\cdots,L_{k-1})$ is unmixed,  we get the desired result by Lemma \ref{Lem: unmixed sep}.
\end{proof}

 Let $\mathcal I_\beta$ be the set of sequences $(\beta_1,\ldots,\beta_h)$  with $h \ge 1$,  $\beta_k \in \Z_{\ge 0} \Delta_+$ for $1 \le k \le h$ 
and $\beta=\sum_{k=1}^h \beta_k$.

We define  a total order on $\mathcal I_\beta$ as follows:
$$
 (\beta_1, \beta_2, \ldots, \beta_p) \lpreceq (\beta_1', \beta_2', \ldots, \beta_{q}')
$$
if one of the following conditions holds:
\begin{enumerate}
\item $(\beta_1, \ldots, \beta_p) = (\beta_1', \ldots, \beta_q')$,
\item there exists a positive integer $k \le  \min \{p,q\} $ such that \\
(a) $\beta_j = \beta_j'$ for any $j < k$, and (b) $\beta_k \prec \beta_k'$,
\item there exists a positive integer $k \le  \min \{p,q\} $ such that \\
(a) $\beta_j = \beta_j'$ for any $j < k$,  (b) $\beta_k =t \beta_k' $ for some
$t < 1$.
\end{enumerate}
Similarly, we define  another total order:
$$
(\beta_1, \beta_2, \ldots, \beta_p) \rpreceq (\beta_1', \beta_2', \ldots, \beta_{q}')
$$
if one of the following conditions holds:
\begin{enumerate}
\item $(\beta_1, \ldots, \beta_p) = (\beta_1', \ldots, \beta_q')$,
\item there exists a positive integer $k \le  \min \{p,q\} $ such that\\
(a) $\beta_{p-j+1} = \beta_{q-j+1}'$ for any $j < k$, and (b) $\beta_{p-k+1} \succ \beta_{q-k+1}'$,
\item there exists a positive integer $k \le  \min \{p,q\} $ such that \\
(a) $\beta_{p-j+1} = \beta_{q-j+1}'$ for any $j < k$, (b) $\beta_{p-k+1} = t \beta_{q-k+1}'$ for some $t < 1$.
\end{enumerate}
Now we define
$$(\beta_1, \beta_2, \ldots, \beta_p) \bpreceq (\beta_1', \beta_2', \ldots, \beta_{q}')$$
if
$(\beta_1, \beta_2, \ldots, \beta_p) \lpreceq (\beta_1', \beta_2', \ldots, \beta_{q}')$
and
$(\beta_1, \beta_2, \ldots, \beta_p) \rpreceq (\beta_1', \beta_2', \ldots, \beta_{q}')$.

For a simple $R(\beta)$-module $L$ with  $\cd(L) = (L_1, L_2, \ldots, L_p)$,
we set
\[
\cA_{\preceq}(L) \seteq (-\wt(L_1), \ldots, -\wt(L_p)).
\]
 For simple $R(\beta)$-modules, $L$ and $L'$
we define
\eqn
 &&\cd(L) \lpreceq \cd(L') \quad \text{if } \  \cA(L) \lpreceq \cA(L'),  \\
 &&\cd(L) \rpreceq \cd(L') \quad \text{if } \  \cA(L) \rpreceq \cA(L'),  \ \text{and}\\
 &&\cd(L) \bpreceq \cd(L') \quad \text{if } \  \cA(L) \bpreceq \cA(L').
 \eneqn
Note that $\lpreceq$ and $\rpreceq$ are total preorders on the set of $\preceq$-cuspidal decompositions of all simple $R$-modules.

\begin{prop} \label{Prop: bi-preorder}
 Let $L$ be a simple $R(\beta)$-module with $ \cd(L) = (L_1, \ldots, L_h)$ and
 let $\beta_k = - \wt(L_k)$ for $k=1, \ldots,h$.
\bnum
\item $L$ appears once in $L_1 \conv \cdots \conv L_h$.
\item If $\Res_{\beta_1',\ldots,\beta_t'}  (L_1\conv \cdots \conv L_h)\neq 0$
for some  $(\beta_1',\ldots,\beta_t') \in \mathcal I_\beta$,
then we have
$$(\beta_1',\ldots,\beta_t') \bpreceq (\beta_1,\ldots,\beta_h).$$
\item If $L'$ is a simple subquotient of $L_1 \conv \cdots \conv L_h$ which is not isomorphic to $L$, then we have
$ \cd(L') \bprec  \cd(L)$.
\end{enumerate}
\end{prop}
\begin{proof}

\smallskip\noi
(i) follows from Lemma \ref{Lem: unmixing}.

\smallskip\noi
(ii)\
 Assume that
\[
(\beta_1, \ldots, \beta_h) \ne (\beta_1', \ldots, \beta_t').
\]

 We shall show that $(\beta_1',\ldots,\beta_t') \lprec (\beta_1,\ldots,\beta_h)$.

\smallskip\noi
\textbf{Case (l-a)}\ We first consider the case $\beta_1 \ne \beta_1'$.
Since $\beta_1'  \in  \gW(L_1 \conv \cdots \conv L_h) , $
Lemma \ref{Lem: sum of W} allows us to write
\[
\beta_1' = \gamma_1 + \gamma_2 \quad \text{ for some $\gamma_1 \in \gW(L_1)$ and some $ \gamma_2 \in \gW(L_2 \conv \cdots \conv L_h)$.}
\]
Note that, by  Lemma \ref{Lem: sum of W},
\[
\gW(L_2 \conv \cdots \conv L_h) \subset \Sp\{ \gamma \in \prD \mid \gamma \prec \beta_1 \}.
\]

If $\gamma_2 \ne 0$, then we have
$ \beta_1' \prec \beta_1 $ by Lemma \ref{Lem: convex2}. Thus we have $(\beta_1',\ldots,\beta_t') \lpreceq (\beta_1,\ldots,\beta_h)$.

If $\gamma_2 = 0$, then $\beta_1' \in \gW(L_1)$, which tells $ \Ht(\beta_1') < \Ht(\beta_1) $ since $\beta_1' \ne \beta_1$.
If $\beta_1' \prec \beta_1$, then we have $(\beta_1',\ldots,\beta_t') \lprec (\beta_1,\ldots,\beta_h)$.  Otherwise, $\beta_1$ and $\beta_1'$ are in the same $\preceq$-equivalence class.
Thus, $\beta_1' = t \beta_1$ for some $t<1$, which implies $(\beta_1',\ldots,\beta_t') \lprec (\beta_1,\ldots,\beta_h)$.

\smallskip\noi
\textbf{Case (l-b)}\ We now suppose that there is $k \in  \Z_{>0}$ such that $\beta_j = \beta_j'$ for any $j<k$ and $\beta_k \ne \beta_k'$.
Let $\gamma = \beta - \sum_{j=1}^{k-1} \beta_j$.
By Lemma \ref{Lem: unmixed conv} we have
\[
\Res_{\beta_1, \ldots, \beta_{k-1}, \gamma} (L_1 \conv \cdots \conv L_h ) \simeq L_1  \otimes \cdots \otimes L_{k-1} \otimes( L_k \conv \cdots \conv L_h).
\]
Then  the assumption
$\Res_{\beta_1',\ldots,\beta_t'}  (L_1 \conv \cdots \conv L_h)\neq 0$
implies
\[ \Res_{\beta_k',\ldots,\beta_h'} ( L_k \conv \cdots \conv L_h) \neq 0.
\]
Therefore,  by \textbf{Case (l-a)},
we conclude that $(\beta_k',\ldots,\beta_t') \lprec (\beta_k,\ldots,\beta_h)$, which yields $(\beta_1',\ldots,\beta_t') \lprec (\beta_1,\ldots,\beta_h)$.

We now shall prove that $(\beta_1',\ldots,\beta_t') \rprec (\beta_1,\ldots,\beta_h)$. The proof is similar to
the case of $\lprec$.

\noi
\textbf{Case (r-a)}\ Suppose that $\beta_h \ne \beta_t'$. Then, in the same manner as above, we have
\begin{align*}
\sgW(L_h) &\subset \Sp\{ \gamma \in \prD \mid \beta_h \preceq \gamma  \},\\
\sgW( L_1\conv \cdots \conv L_{h-1} ) & \subset \Sp\{ \gamma \in \prD \mid \beta_h \prec \gamma  \}.
\end{align*}
We write $ \beta_t' = \gamma_1 + \gamma_2$ for some $ \gamma_1 \in \sgW(L_h) $ and some $ \gamma_2 \in \sgW( L_1\conv \cdots \conv L_{h-1} )$.

If $\gamma_2\ne 0$, then $\beta_h \prec \beta_t'$ by Lemma \ref{Lem: convex2}.
Thus we have $(\beta_1',\ldots,\beta_t') \rprec (\beta_1,\ldots,\beta_h)$.

If $\gamma_2 = 0$, then $\beta_t' \in \sgW(L_h)$, which tells $ \Ht(\beta_t') < \Ht(\beta_h) $ since $\beta_t' \ne \beta_h$.
If $\beta_t' \succ \beta_h$, then we have $ \cd(L') \rprec \cd(L)$. Otherwise, $\beta_h$ and $\beta_t'$ are in the same $\preceq$-equivalence class.
Thus, $\beta_t' = p \beta_h$ for some $p<1$, which implies $(\beta_1',\ldots,\beta_t') \rprec (\beta_1,\ldots,\beta_h)$.

\smallskip\noi
\textbf{Case (r-b)}\ We now assume that there is $k \in \Z_{>0}$ such that $ \beta_{h-j+1} = \beta_{t - j+1}'$ for any $j < k$ and $ \beta_{h-k+1} \ne \beta_{t-k+1}$.
 Let $\gamma = \beta - \sum_{j=1}^{k-1} \beta_{h-j+1}$.
Since
\[
\Res_{\gamma, \beta_{h-k+2}, \ldots, \beta_h} (L_1 \conv \cdots \conv L_{h} ) \simeq (L_1 \conv  \cdots \conv L_{h-k+1})
\otimes L_{h-k+2} \otimes \cdots \otimes L_h,
\]
 the assumption
$\Res_{\beta_1',\ldots,\beta_t'}  (L_1 \conv \cdots \conv L_h)\neq 0$
implies
\[ \Res_{\beta_1',\ldots,\beta_{t-k+1}'} ( L_1 \conv \cdots \conv L_{h-k+1}) \neq 0.
\]
 By \textbf{Case (r-a)}, we can conclude that
$ ( \beta_1' , \ldots , \beta_{t-k+1}' )   \rprec (\beta_1 , \ldots , \beta_{h-k+1} ) $, which yields that
$(\beta_1',\ldots,\beta_t') \rprec (\beta_1,\ldots,\beta_h)$.

\smallskip\noi
 (iii) \
Let $\cd(L')=(L_1',\ldots,L_t')$ and $\cA(L')=(\beta_1',\ldots,\beta_t')$.
Since $\Res_{\beta_1'\ldots,\beta_t'}L' \neq 0$, we have $\Res_{\beta_1'\ldots,\beta_t'}(L_1\conv \cdots \conv L_h) \neq 0$.
It follows from (ii) that $(\beta_1',\ldots,\beta_t') \bpreceq (\beta_1,\ldots,\beta_h)$.
Note that $\Res_{\beta_1,\ldots,\beta_h}L'=0$ by Theorem \ref{Thm: cuspidal decomposition} and hence
$(\beta_1, \ldots, \beta_h) \ne (\beta_1', \ldots, \beta_t').$
Thus we conclude that
$\cd(L') \bprec \cd(L)$.
\end{proof}

\subsection{Categories $\catC_{w,v}$}

For $w\in \weyl$, we denote by $\catC_{w}$ the  full subcategory of $R\gmod$ whose objects $M$ satisfy
\begin{align} \label{Eq: conditions of C_w}
\gW(M) \subset \Sp( \prD \cap w \nrD ).
\end{align}
Similarly, for $v\in \weyl$, we define $\catC_{*,v}$ to be the  full subcategory of $R\gmod$ whose objects $N$ satisfy
\begin{align} \label{Eq: conditions of C_*v}
\sgW(N) \subset \Sp( \prD \cap v \prD ).
\end{align}
For $M \in \catC_{w}$ and $N \in \catC_{*,v}$,
it is obvious that
\begin{align*} 
\gW(M) \subset \rlQ_+ \cap w\rlQ_-,  \quad  \sgW(N) \subset \rlQ_+ \cap v\rlQ_+ ,
\end{align*}
which implies
\begin{align*} 
\gW(M) \cap \prD \subset \prD \cap w \nrD,  \quad \sgW(N) \cap \prD \subset \prD \cap v \prD.
\end{align*}
By Proposition \ref{Prop: Span W(L)},
we have the following equivalences:
\begin{equation}\label{Eq: gW and sgW}
\begin{aligned}
M \in \catC_{w} \Longleftrightarrow \gW(M) \subset \rlQ_+ \cap w\rlQ_- \Longleftrightarrow\gW(M) \cap \prD \subset \prD \cap w \nrD, \\
N \in \catC_{*,v} \Longleftrightarrow \sgW(N) \subset \rlQ_+ \cap v\rlQ_+ \Longleftrightarrow \sgW(N) \cap \prD \subset \prD \cap v \prD.
\end{aligned}
\end{equation}

For $w, v\in \weyl$, we define
$\catC_{w,v} $ to be the full subcategory of $R\gmod$ whose objects are contained in both of the subcategories $ \catC_w$ and $\catC_{*,v}$.
By the construction, we have the following proposition.
\begin{prop} \label{Prop: closedness}
The categories $\catC_w$, $\catC_{*,v}$ and $\catC_{w,v}$ are stable under taking subquotients, extensions, convolution products and grading shifts.
 In particular, their Grothendieck groups
are $\Z[q,q^{-1}]$-algebras.
\end{prop}

\begin{lem} \label{Lem: ep and ep*}
Let $w \in \weyl$ and $i\in I$.
\bnum
\item If $s_i w > w$, then $\ep_i(L) = 0$ for any $L \in \catC_{w}$.
\item If $s_i w < w$, then $\ep^*_i(L) = 0$ for any $L \in \catC_{*,w}$.
\end{enumerate}
\end{lem}
\begin{proof}
(i)\ Suppose that $L \in \catC_w$ and $s_iw > w$. Then $\alpha_i \in \prD \cap w \prD$, which tells that
$ \alpha_i \notin  \rtl_+\cap w \rtl_-$.
Thus, we have $\alpha_i \notin \gW(L)$.

\smallskip\noi
(ii)\ Suppose that $L \in \catC_{*,w}$ and $s_iw < w$. Then $\alpha_i \in \prD \cap w\nrD$, so we have
$ \alpha_i \notin \rtl_+\cap w\rtl_+$.
This implies $\alpha_i \notin \sgW(L)$.
\end{proof}

The following proposition gives a membership condition on the categories using the cuspidal decomposition.
\begin{prop} \label{Prop: membership}
Let $\underline{w} = s_{i_1}s_{i_2} \cdots s_{i_\ell}$ be a reduced expression of $w\in \weyl$. We denote by $\preceq$
a convex order on $\prD$ which refines the convex preorder with respect to $\underline{w}$
given in {\rm Proposition~\ref{Prop: convex preorder for w}}, and set
$\beta_\ell = s_{i_1} \cdots s_{i_{\ell-1}}(\alpha_{i_\ell})$.
We take a simple $R$-module $L$ and set
\[
\cd(L) \seteq (L_1, L_2, \ldots, L_h),\quad \gamma_k \seteq -\wt(L_k)\quad \text{for }k=1, \ldots, h.
\]
Then we have
\bnum
\item $L \in \catC_{w}$ if and only if  $\beta_\ell \succeq \gamma_1$,
\item $L \in \catC_{*, w}$ if and only if $\gamma_h   \succ  \beta_\ell$.
\end{enumerate}
\end{prop}
\begin{proof}
(i)\ Suppose that $\beta_\ell \succeq \gamma_1$. As $\gamma_1$ is largest among $\gamma_k$'s,  we have
\[
\gW(L_1 \conv L_2 \conv \cdots \conv L_h) \subset \Sp(\prD \cap w \nrD)
\]
by Lemma \ref{Lem: sum of W}. Thus we have $L \in \catC_{w}$.

We assume that $L \in \catC_{w}$. By Lemma \ref{Lem: unmixing}, we have
\[
\gamma_1 \in \gW(L) \subset \Sp(\prD \cap w \nrD).
\]
Since $\beta_\ell$ is largest in $\prD \cap w \nrD$,
we have $\beta_\ell \succeq \gamma_1$.

\snoi
(ii)\ Suppose that $\gamma_h \succ \beta_\ell$. Since $\gamma_h$ is smallest among $\gamma_k$'s,  Lemma \ref{Lem: sum of W} tells
\[
\gW^*(L_1\conv L_2 \conv \cdots \conv L_h) \subset \Sp(\prD \cap w \prD),
\]
which implies $L \in \catC_{*,w}$.

We now assume that $L \in \catC_{*, w}$. It follows from Lemma \ref{Lem: unmixing} that
\[
\gamma_h \in \gW^*(L) \subset \Sp(\prD \cap w \prD).
\]
Since any positive root in $\prD \cap w \prD$ is larger than $\beta_\ell$,
we obtain $ \gamma_h \succ \beta_\ell$.
\end{proof}

We now revisit the unipotent quantum coordinate ring $\Aqq[\n]$. For $w \in \weyl$, let $ A_{w}$ (resp.\ $A_{*,w}$) be
the $\Zq$-linear subspace of $\Aqq[\n]$ spanned by all elements $x \in \Aqq[\n]$ such that
\begin{align*}
e_{i_1} \cdots e_{i_l}x=0 \quad \text{(resp.\ $e_{i_1}^* \cdots e_{i_l}^*x=0 $)}
\end{align*}
for any sequence $(i_1, \ldots, i_l) \in I^\beta$ with $\beta \in \rlQ_+ \cap w\rlQ_+\rmz$
(resp.\ $\beta \in \rlQ_+ \cap w\rlQ_-\rmz$).

For $w,v \in \weyl$, we define
\[
A_{w,v} \seteq A_w \cap A_{*,v} \subset \Aqq[\n].
\]

\begin{rem} \label{Rmk: geo}
Let $G$ be a  reductive  algebraic group over $\C$ and $\g$ its Lie algebra.
We fix a maximal torus $H$ of $G$ and take a Borel subgroup $B$ and its opposite Borel subgroup $B^-$  containing  $H$.
We denote by $N$ and $N^-$ the unipotent radicals of $B$ and $B^-$ respectively, and $\n$ and $\n^-$ denote the Lie algebras of $N$ and $N^-$.
Let $\weyl\seteq {\rm Norm}_G (H) / H$ be the Weyl group.
Let $\C[N]$ be the coordinate ring of $N$ and, for $w, v \in \weyl$,
\[
N'(w) \seteq N \cap ( w N w^{-1} ) \quad \text{ and } \quad  N(v) \seteq N \cap ( v N^- v^{-1} ).
\]
Note that $\C[N]$ is isomorphic to the dual of $U(\n)$. We consider the doubly-invariant  algebra
\[
^{N'(w)}\C[N]^{N(v)} \seteq \{ f \in \C[N] \mid f(n x m) = f(x) \text{ for all $x \in N$,  $m \in N'(w),n \in N(v)$}   \}
\]
which was introduced in \cite{Lec15}.
Then we have
\begin{align*}
^{N'(w)}\C[N]^{N(v)} &= \{ f \in \C[N] \mid x f y = 0 \text{ for all $x \in \n \cap w\n, y \in \n \cap v \n^-$}   \}.
\end{align*}
By the PBW theorem, we have
\[
U(\n) =U(\n \cap u\n^-)\, U(\n \cap u \n)\qt{for any $u \in \weyl$.}
\]
Hence, we have
\eqn
\sum_{\beta \in \rlQ_+ \cap w \rlQ_+\rmz}U(\n)_\beta&\subset&
\sum_{\substack{\beta \in \rlQ_+ \cap w \rlQ_+\rmz,\\
\gamma\in w\rlQ_-}}
U(\n \cap w\n^-)_\gamma U(\n \cap w\n)_{\beta-\gamma}\\
&\subset&
\sum_{\al\in \rlQ_+\cap w\rlQ_+\rmz} U(\n \cap w\n^-) U(\n \cap w\n)_\al\,.
\eneqn
Similarly, we have
\eqn
\sum_{\beta \in \rlQ_+ \cap v \rlQ_-\rmz}U(\n)_\beta&\subset&
\sum_{\substack{\beta \in \rlQ_+ \cap v \rlQ_-\rmz,\\\gamma\in v\rlQ_+}}
 U(\n \cap v\n^-)_{\beta-\gamma} U(\n \cap v\n)_{\gamma}\\
&\subset &
\sum_{\al\in \rlQ_+\cap v\rlQ_-\rmz} U(\n \cap v\n^-)_\al U(\n \cap v\n).
\eneqn
Thus, for $f \in {^{N'(w)}\C[N]^{N(v)}}$,  $\beta \in \rlQ_+ \cap w \rlQ_+ \setminus \{0 \}$ and $ \gamma \in  \rlQ_+ \cap v \rlQ_- \setminus \{0 \}$, we have
\begin{align*}
U(\n)_\beta f U(\n)_\gamma = 0,
\end{align*}
which implies that
{\allowdisplaybreaks[1]
\eqn
^{N'(w)}\C[N]^{N(v)} =\bigl \{ f \in \C[N] \mid &&
\text{$U(\n)_\beta f U(\n)_\gamma = 0$}\\
&&\hs{.5ex}\text{for any $\beta \in \rlQ_+ \cap w \rlQ_+ \setminus \{0 \}$
and $\gamma \in  \rlQ_+ \cap v \rlQ_- \setminus \{0 \}$}  \bigr\}.
\eneqn
}

This tells that the ring $^{N'(w)}\C[N]^{N(v)}$ is the subspace of $\C[N]$ consisting of all elements $f$ such that
\begin{align*}
e_{i_1} \cdots e_{i_l}f=0 \quad \text{and} \quad e_{j_1}^* \cdots e_{j_t}^*f=0
\end{align*}
for all $(i_1, \ldots, i_l) \in I^\beta$ with $\beta \in \rlQ_+ \cap w\rlQ_+\rmz$  and all $ (j_1, \ldots, j_t) \in I^\gamma $
with $\gamma \in \rlQ_+ \cap v\rlQ_-\rmz$.
Therefore, $A_{w,v}$ can be thought of as a quantum deformation of $^{N'(w)}\C[N]^{N(v)}$.
\end{rem}

\begin{thm} \label{Thm: A and catC}  Let $w,v \in \weyl$.
\bnum
\item
Let $x \in \Aqq[\n]$. 
\be[{\rm(1)}]
\item The following conditions are equivalent:
\bna
\item $e_{i_1} \cdots e_{i_l}x=0$
for any $(i_1, \ldots, i_l) \in I^\beta$ with $\beta \in \rlQ_+ \setminus  w \rlQ_-$,
\item $e_{i_1} \cdots e_{i_l}x=0$
for any $(i_1, \ldots, i_l) \in I^\beta$ with $\beta \in \rlQ_+ \cap w\rlQ_+\rmz$,
\item $e_{i_1} \cdots e_{i_l}x=0$
for any $(i_1, \ldots, i_l) \in I^\beta$ with  $\beta \in \prD \cap w\prD$.
\end{enumerate}

\item The following conditions are equivalent:
\bna
\item $e_{i_1}^* \cdots e_{i_l}^*x=0$
for any $(i_1, \ldots, i_l) \in I^\beta$ with $\beta \in \rlQ_+ \setminus  v \rlQ_+$,
\item $e_{i_1}^* \cdots e_{i_l}^*x=0$
for any $(i_1, \ldots, i_l) \in I^\beta$ with $\beta \in \rlQ_+ \cap  v \rlQ_-\rmz$,
\item $e_{i_1}^* \cdots e_{i_l}^*x=0$
for any $(i_1, \ldots, i_l) \in I^\beta$ with $\beta \in \prD \cap v\nrD$.
\end{enumerate}
\end{enumerate}

\item
Under the identification via the isomorphism $ K_0( R\gmod ) \simeq \Aqq[\n]$, we have
\bna
\item   $K_0(\catC_w) = A_{w}$,
\item $ K_0(\catC_{*,v}) = A_{*,v},$
\item   $K_0(\catC_{w,v}) = A_{w,v} $.
\end{enumerate}
\end{enumerate}
\end{thm}
\begin{proof}
We first deal with the case (i) (1) and (ii) (a) together.
\begin{align*}
A & := \{ x \in { \Aqq[\n]}  \mid   e_{i_1} \cdots e_{i_l}x=0    \text{ for any } (i_1, \ldots, i_l) \in I^\gamma \text{ with } \gamma \in \rlQ_+ \setminus  w \rlQ_- \}, \\
B  & := \{ x \in \Aqq[\n]  \mid   e_{i_1} \cdots e_{i_l}x=0    \text{ for any } (i_1, \ldots, i_l) \in I^\gamma \text{ with } \gamma \in \rlQ_+ \cap w\rlQ_+ \}, \\
C & := \{ x \in \Aqq[\n]  \mid   e_{i_1} \cdots e_{i_l}x=0
\text{ for any } (i_1, \ldots, i_l) \in I^\gamma \text{ with } \gamma \in \prD \cap w\prD \}, \\
K  & := K_0(\catC_w).
\end{align*}
We shall show that all of four spaces coincide.
By the definition, we have
\[
A \subset B\subset C.
\]
Thus it suffices to show that $ K\subset A$ and $ C\subset K$.

\snoi
\textbf{Inclusion $K\subset A$:}\
Let $M$ be a simple $R(\beta)$-module in $\catC_w$. We take a sequence $(i_1, \ldots, i_l) \in I^\gamma$ with  $\gamma  \in \rlQ_+ \setminus w\rlQ_-$.
By $\eqref{Eq: gW and sgW}$, we have
\[
\Res_{ \gamma, \beta - \gamma } M =0,
\]
which implies that $e_{i_1} \cdots e_{i_l} [M] = 0$. Thus, we have $[M] \in A$.

\snoi
\textbf{Inclusion $ C \subset K$:}\
Let $\preceq$ be a convex order on $\prD$ which refines the convex preorder corresponding to a reduced expression of $w$ given in Proposition \ref{Prop: convex preorder for w}.

Let  $f$ be an element in $C$.
We shall show that $f\in K$.
By Theorem \ref{Thm: categorification}, $f$ can be written as
\[
f = \sum_{k=1}^m a_k [M_k]
\]
for some non-zero $a_k \in \Z[q,q^{-1}]$ and some
self-dual simple $R$-modules $M_k$
which are not isomorphic to each other.
We may assume $m>0$. Without loss of generality, we may assume that
\[
\cd(M_{j}) \lsucceq \cd(M_{j+1}) \qquad \text{ for $1 \le j<m$.}
\]
We write $\cd(M_1) = (L_1, L_2, \ldots, L_h)$ and let $\gamma_k = -\wt(L_k)$ for $k=1,\ldots, h$. Note that
\begin{align} \label{Eq: W(Mk)}
\gW(M_j) \subset \Sp\{ x \in \prD \mid \gamma_1 \succeq x   \} \qquad \text{ for $j = 1, \ldots, m$}.
\end{align}
It follows from Lemma \ref{Lem: unmixing} and Proposition \ref{Prop: bi-preorder}  that, for $j=1, \ldots, m$,
\begin{align} \label{Eq: Res_gamma}
\Res_{\gamma_1, \ldots, \gamma_h}(M_j) \simeq
\left\{
	\begin{array}{ll}
		0 & \hbox{ if } \cd(M_j) \lprec \cd(M_1), \\
		L_{j,1} \otimes \cdots \otimes L_{j,h_j} & \hbox{otherwise},
	\end{array}
\right.
\end{align}
where $\cd(M_j) =  (L_{j,1} , \ldots , L_{j,h_j})$.

By $\eqref{Eq: preceq and prD}$, we can write $\gamma_1 =  n \alpha$ for some $ n \in \Z_{>0}$ and some $\alpha \in \prD^{\min}$.

\snoi
(a)\ Suppose that $\alpha \in \prD \cap w \nrD$.
By $\eqref{Eq: W(Mk)}$ and the definition of the convex order $\preceq$,   we have
\[
\gW(M_j) \subset \Sp\{ \beta' \in \prD \mid \beta' \preceq \alpha \} \subset \rlQ_+ \cap w \rlQ_- \quad \text{ for $j=1, \ldots, m$,}
\]
which implies $ f \in K$.

\snoi
(b) Suppose that $\alpha \in \prD \cap w \prD$.

\noi
(b\/1) Assume first that $\al$ is an imaginary root.
Then, $\gamma_1=c\al\in \prD \cap w \prD$.
Letting $d = \Ht(\gamma_1)$, we have
\[
e_{i_1} \cdots e_{i_{d}} f = 0 \quad \text{for any sequence $(i_1, \ldots, i_{d}) \in I^{\gamma_1}$,}
\]
which implies
\[
0 = \sum_{j=1}^m a_j \ch(\Res_{\gamma_1, \ldots, \gamma_h}M_j).
\]
By $\eqref{Eq: Res_gamma}$, $\Res_{\gamma_1, \ldots, \gamma_h}(M_j)$ is zero or is isomorphic to an outer tensor product of simple modules.
Thus the set
\[
 \{  \ch(\Res_{\gamma_1, \ldots, \gamma_h}(M_j)) \mid \Res_{\gamma_1, \ldots, \gamma_h}M_j \ne 0,\ j=1, \ldots, m  \}
\]
is linearly independent. It follows from $ \Res_{\gamma_1, \ldots, \gamma_h}(M_1) \ne 0$ that
\[
a_1 = 0,
\]
which is a contradiction to the non-vanishing assumption on $a_1$.
Therefore we conclude that $f\in K$.

\snoi
(b$\ms{.5mu}$2) Assume now that $\al$ is a real root.
We set $\mu=-\wt(f)=-\wt(M_j)\in\rlQ_+$.
We set $\beta_j=-\wt(L_{j,1})$ where
we set $\cd(M_j) =  (L_{j,1} , \ldots , L_{j,h_j})$.
 We set $\set{j}{\beta_j=\beta_1=n\al}=[1,b]$
with $1\le b\le m$.
Then for $j>b$ we have either
$\beta_j=m\al$ for some $m<n$ or $\beta_j\prec \al$.

Then by 
Proposition \ref{Prop: bi-preorder},
we have
$$\Res_{n\al,\mu-n\al}(M_j)
\simeq\bc L_{j,1}\tens\hd(L_{j,2}\conv\cdots\conv L_{j,h_j} )
&\text{if $j\in[1,b]$,}\\
0&\text{otherwise.}
\ec
$$
By Proposition~\ref{prop:realcuspidal}, we have
$L_{j,1}\simeq L(n\al)$ for $j\in[1,b]$.
Here $L(\al)$ is the $\preceq$-cuspidal simple $R(\al)$-module
and $L(n\al)\simeq L(\al)^{\circ n}$ up to
grading shift.
Let us take $(i_1, \ldots, i_{d}) \in I^{\al}$
such that
$E_{i_1}\cdots E_{i_d}L(\al)$ does not vanish.
Then we have
\begin{equation}\label{Eq: constant}
[(E_{i_1}\cdots E_{i_d})^nL(n\al)]=c
\end{equation}
for some
$c\in\Z_{\ge0}[q,q^{-1}]\rmz$.
Thus we obtain
\eqn
&&(e_{i_1}\cdots e_{i_d})^nf=
\sum_{j=1}^ma_j[(E_{i_1}\cdots E_{i_d})^nM_j]
=c\sum_{j=1}^ba_j[\hd(L_{j,2}\conv\cdots\conv L_{j,h_j})].
\eneqn
Since $\{[\hd(L_{j,2}\conv\cdots\conv L_{j,h_j})]\}_{j\in[1,b]}$ is linearly independent,
it contradicts $(e_{i_1}\cdots e_{i_d})f=0$.

Thus we obtain $A=B=C=K$.

\bigskip
 In a similar manner as above, one can prove (i) (2) and (ii) (b).

Assertion (ii) (c) follows immediately from (ii) (a) and (ii) (b).
\end{proof}

\begin{rem}
\bnum
\item Since $\Aq[\n]$ is the dual of  $U_q^+(\g)$,
{\rm Theorem \ref{Thm: A and catC} (i)} implies that
\eq
&&\hs{5ex}\sum_{\beta\in\rtl_+\setminus w\rtl_{\mp}}U_q^+(\g)U_q^+(\g)_\beta
=\sum_{\beta\in\rtl_+\cap w\rtl_{\pm}\rmz}U_q^+(\g)U_q^+(\g)_\beta
=\sum_{ \beta\in\prD \cap w\Delta_{\pm}  }U_q^+(\g)U_q^+(\g)_\beta.
\label{eq:Uwv}
\eneq
\item
By specializing at $q=1$, the same properties as in
{\rm Theorem \ref{Thm: A and catC} (i)} hold for $\C[N]^{N(v)}$, $^{N'(w)}\C[N]$ and $^{N'(w)}\C[N]^{N(v)}$.
Indeed, the same proof of Theorem \ref{Thm: A and catC} can be applied
by replacing the categories of graded modules
over quiver Hecke algebras with their non-graded versions.
The constant $c$ in \eqref{Eq: constant} becomes then a positive integer.
 Note that the analogue of
\eqref{eq:Uwv} at $q=1$  still holds.
\ee
\end{rem}

\begin{cor}
The subspaces $ A_{w}, A_{*,v}$ and $A_{w,v}$ are subalgebras of $\Aqq[\n]$.
\end{cor}
\begin{proof}
It follows from Proposition \ref{Prop: closedness} and Theorem \ref{Thm: A and catC}.
\end{proof}

\begin{rem}
Suppose that $\cmA$ is of finite type. Then there is the longest element $w_0 \in \weyl$ and
\begin{align*}
\rlQ_+\cap w \rlQ_- = \rlQ_+\cap w w_0 \rlQ_+, \qquad \rlQ_+\cap v \rlQ_+ = \rlQ_+\cap v w_0 \rlQ_-
\end{align*}
for $w,v \in \weyl$.
Thus, the algebra involution $*$ of $R(\beta)$ gives an equivalence of two categories
$\catC_{w,v}$ and $\catC_{vw_0, ww_0} $, where
$*: R(\beta) \to R(\beta)$ is given by
\eqn
e(\nu) \mapsto  e(\overline{\nu}), \quad
x_k \mapsto   x_{n-k+1}, \quad
\tau_l\mapsto -\tau_{n-l}.
\eneqn
Here  $n=\Ht(\beta)$ and $\overline{\nu}$ denotes the reversed sequence $(\nu_n, \nu_{n-1}, \ldots, \nu_2, \nu_1)$ of $\nu=(\nu_1,\ldots,\nu_n)$.
\end{rem}

\vskip 2em

\section{R-matrices} \label{Sec: R-matrix}

\subsection{Strongly commuting pairs}

In this subsection, we assume that the quiver Hecke algebra $R$ is arbitrary.

For simple $R$-modules $M$ and $N$,
we say that $M$ and $N$ \emph{strongly commute}
or $M$ \scts with $N$ if $M\conv N$ is simple.
We say that a simple $R$-module $L$ is \emph{real} if
$L$ \scts with itself, i.e., if $L \conv L$ is simple.

For $M \in R(\beta)\gmod$, we define
$M^\dual \seteq \HOM_\bR( M, \bR)$, which admits an $R(\beta)$-module structure via
$$
 (r \cdot f) (u) \seteq f( \psi(r) u) \qquad \text{for
 $f\in M^\dual$,  $r \in R(\beta)$ and $u\in M$},
$$
where $\psi$ denotes the anti-automorphism of  $R(\beta)$ which fixes the generators.
We say that a simple module $L$ is {\em self-dual}
if $L\simeq L^\dual$.
It is known that
\begin{itemize}
\item for a simple $R$-module $L$,
there exists $n\in\Z$ such that $q^n L$ is self-dual,
\item for $M,N \in R\gmod$,
$
N^{\dual } \conv M^\dual  \simeq q^{-( \wt(M), \wt(N))}(M \conv N)^\dual.
$
\end{itemize}
Thus, if two simple modules $M$ and $N$ \sct, then  we have
\[
M \conv N \simeq  N \conv M
\]
up to a grading shift.

\begin{lem} \label{Lem: com tE tEs}
Let $M$ and $N$ be simple $R$-modules. Suppose that $M$ and $N$ \sct. Then we have
\bnum
\item $\tEm_i M$ and $\tEm_i N$ \sct,
\item $\tEsm_i M$ and $\tEsm_i N$ \sct.
\end{enumerate}
\end{lem}
\begin{proof}
(i) Letting $m = \ep_i(M)$ and $ n = \ep_i(N)$, we get
$ E_i^{(m+n)} (  M\conv N ) \simeq E_i^{(m)}M \conv E_i^{(n)}N $ up to a grading shift and it is simple.

\noi
(ii) can be proved in the same manner as above  by replacing
the role of $E_i$ with the one of $E_i^*$.
\end{proof}

\begin{lem} \label{Lem: comm with L(i)}
Let $M$ and $N$ be simple $R$-modules and $i\in I$.
\bnum
\item If $M$ \scts with $L(i)$, then $\ep_i(M \hconv N) = \ep_i(M) + \ep_i(N)$.
\item If $N$ \scts with $L(i)$, then $\ep_i^*(M \hconv N) = \ep_i^*(M) + \ep_i^*(N)$.
\end{enumerate}
\end{lem}
\begin{proof}
(i) We set
$M_0 \seteq E_i^{(m)}M $ and $N_0 \seteq E_i^{(n)}N $, where $m = \ep_i(M) $ and $n = \ep_i(N)$.
Then we have
\[
L(i^{m+n}) \conv M_0 \conv N_0 \twoheadrightarrow L(i^n) \conv M \conv N_0
\simeq  M \conv L(i^n) \conv N_0 \twoheadrightarrow M \conv N
\twoheadrightarrow M \hconv N,
\]
up to grading shifts, which implies that $\ep_i(M \hconv N) = m+n$.

\noi
(ii) can be proved in the same manner as (i).
\end{proof}

\subsection{Normalized R-matrices} \label{Sec: subsection R-matrices}

We recall the notions of affinizations and $R$-matrices for symmetric quiver Hecke algebras introduced in \cite{KKK}.
Note that they  have been generalized to arbitrary quiver Hecke algebras in \cite{KP16}.

Let $\beta\in\rlQ_+$ and $m=\Ht(\beta)$.
For $k = 1, \ldots, m-1$ and $\nu\in I^\beta$, the {\em intertwiner} $\varphi_k \in R(\beta)$ is defined by
$$ \varphi_k e(\nu) \seteq \left\{
              \begin{array}{ll}
                (\tau_kx_k - x_k\tau_k) e(\nu) & \hbox{ if } \nu_k = \nu_{k+1}, \\
                \tau_k e(\nu) & \hbox{ otherwise.}
              \end{array}
            \right.
$$
Then $\{\vphi_k\}_{1\le k\le m-1}$ satisfies the braid relation. Hence,
for $w\in\sg_m$,
the element $\vphi_w\seteq\vphi_{i_1}\cdots\vphi_{i_\ell}$
does not depend on the choice of a reduced expression $w=s_{i_1}\cdots s_{i_\ell}$
of $w$.

Let $M$ be an $R(\beta)$-module with $\Ht(\beta)=m$ and $N$ an $R(\beta')$-module with $ \Ht(\beta')=n$.
Then the $R(\beta)\otimes R(\beta')$-linear map $M\otimes N \longrightarrow N \conv M$ defined by $u\otimes v \mapsto \varphi_{w[n,m]}(v \otimes u)$ can be extended to the $R(\beta+\beta')$-module
homomorphism (up to a grading shift)
\begin{align*}
R_{M,N}: M\conv N \longrightarrow N \conv M.
\end{align*}

\begin{df} \label{Def: symmetric}
The quiver Hecke algebra $R(\beta)$ is said to be \emph{symmetric} if $\qQ_{i,j}(u,v)$ is a polynomial in $u-v$ for any
$i,j \in I$.
\end{df}
We now assume that $R(\beta)$ is symmetric.
Then the generalized Cartan matrix $\cmA$
is symmetric.
In this case we assume
\eqn &&(\al_i,\al_i)=2\qt{for any $i\in I$.}
\eneqn
Let $z$ be an indeterminate. For an $R(\beta)$-module $M$, the $R(\beta)$-module structure on the \emph{affinization} $M_z \seteq \bR[z]\otimes_\bR M$ of $M$
is defined by
\begin{align*}
e(\nu) (f \otimes m) = f \otimes e(\nu)m, \ x_j (f \otimes m) = (zf) \otimes m + f \otimes x_j m, \  \tau_k(f \otimes m) = f \otimes (\tau_k m)
\end{align*}
for $f \in \bR[z]$, $m \in M$, $\nu \in I^\beta$ and admissible $j,k$.
For non-zero $R$-modules $M$ and $N$, we set
$$
\Rm_{M_{z},N_{z'}} \seteq (z' - z)^{-s} R_{M_{z},N_{z'}}:  M_z \conv N_{z'} \to N_{z'}\conv M_z,
$$
where $s$ is the largest integer such that $R_{M_z,N_{z'}}(M_z \conv N_{z'})\subset (z'-z)^s N_{z'}\conv M_z$.
We call it the {\em normalized $R$-matrix}.

Then, we have the intertwiner (up to a grading shift)
$$
\rmat{M,N} : M\conv N\to N\conv M
$$
induced from $\Rm_{M_z,N_{z'}}$ by specializing at $z=z'=0$, which never vanishes by the definition.

\begin{df}
Let $M$ and $N$ be simple $R$-modules. We set
\begin{align*}
\La(M,N) &\seteq \deg (\rmat{M,N}), \\
\tLa(M,N)&\seteq\dfrac{1}{2}\bl \La(M,N)+  \bl \wt(M), \wt(N)   \br  \br, \\
\Dd(M,N) &\seteq \dfrac{1}{2}\bl\La(M,N) + \La(N,M)\br.
\end{align*}
\end{df}

\begin{prop} [\cite{KKKO15}] \label{Prop: R-matrix properties}
Let $M$ and $N$ be simple $R$-modules.
Assume that one of $M$ and $N$ is real.
\bnum
\item $M \conv N$ has a simple head and a simple socle. Moreover,  $ \mathrm{Im} ( \rmat{M,N} )$ is equal to $M\hconv N$ and $N \sconv M$
up to grading shifts.
\item We have
$$
\HOM_R(M\conv N , M \conv N) = \bR   \mathrm{id}_{M\circ N},\quad \HOM_R(M\conv N , N \conv M) = \bR   \rmat{M, N}.
$$
\item If $M\hconv N$ and $M \sconv N$ are isomorphic, then $M$ and $N$ \sct.
\item  $M$ and $N$ \sct if and only if $\Dd(M,N)=0$.
\end{enumerate}
\end{prop}

\begin{lem} \label{Lem: com M N}
Let $M$ and $N$ be simple $R$-modules and let $i\in I$.
Suppose
\bna
\item one of $M$ and $N$ is real,
\item  $M$ and $N$ \sct with $L(i)$.
\ee
Then we have the following.
\bnum
\item If $\tEm_i M$ and $\tEm_i N$ \sct, then $M$ and $N$ \sct.
\item If $\tEsm_i M$ and $\tEsm_i N$ \sct, then $M$ and $N$ \sct.
\end{enumerate}
\end{lem}
\begin{proof}
(i)\ We set
\[
m = \ep_i M, \quad n = \ep_i N, \quad M_0 = \tEm_i M, \quad N_0 = \tEm_i N.
\]
By the assumption, $M_0 \circ N_0$ is simple which tells that $L(i^{m+n}) \conv M_0 \conv N_0$ has a simple head.
Since $M$ and $L(i)$ \sct, we have
\begin{align*}
\hd (L(i^{m+n}) \conv M_0 \conv N_0) & \simeq \hd (L(i^{n}) \conv L(i^{m}) \conv M_0 \conv N_0) \simeq
\hd (L(i^{n}) \conv  M \conv N_0) \\
& \simeq \hd ( M \conv L(i^{n}) \conv  N_0) \simeq  \hd ( M  \conv  N) \simeq M \hconv N,
\end{align*}
up to grading shifts. In the same manner, we have
\[
\hd (L(i^{m+n}) \conv N_0 \conv M_0) \simeq N \hconv M
\]
up to grading shifts.
Thus it follows from $M_0 \circ N_0 \simeq N_0 \circ M_0$ that $M \hconv N \simeq N \hconv M $,
which yields the assertion by Proposition \ref{Prop: R-matrix properties}.

\snoi
(ii) can be proved in the same manner as above by replacing the role of $E_i$ with the one of  $E_i^*$.
\end{proof}

\begin{prop} \label{Prop: formula for La}
Let $M$ and $N$ be simple $R$-modules. Suppose that one of $M$ and $N$ is real.
\bnum
\item Let $m = \ep_i(M)$ and $n = \ep_i(N)$.
If $\ep_i(M\hconv N) = m+n$, then
\begin{align*}
\La(M, N) &= \La(M_0, N_0) - n (\alpha_i, \wt(M_0)) + m (\alpha_i, \wt(N_0))\\
&= \La(M_0, N_0) - n (\alpha_i, \wt(M)) + m (\alpha_i, \wt(N)),
\end{align*}
where $ M_0 = E_i^{(m)} M $ and $ N_0 = E_i^{(n)} N $.
\item Let $m' = \ep^*_i(M)$ and $n' = \ep^*_i(N)$.
If $\ep_i^*(M\hconv N) = m' + n'$, then
\begin{align*}
\La(M, N) &= \La(M_0', N_0') + n' (\alpha_i, \wt(M_0')) - m' (\alpha_i, \wt(N_0')) \\
& = \La(M_0', N_0') + n' (\alpha_i, \wt(M)) - m' (\alpha_i, \wt(N)),
\end{align*}
where  $ M_0' = E_i^{*  (m') } M $ and $ N_0' = E_i^{*  (n') } N $.
\end{enumerate}
\end{prop}

\begin{proof}
Note that
\begin{align*}
 m (\alpha_i, \wt(N_0)) - n (\alpha_i, \wt(M_0)) &=  m (\alpha_i, \wt(N)) - n (\alpha_i, \wt(M)), \\
 m' (\alpha_i, \wt(N_0')) - n' (\alpha_i, \wt(M_0')) &=  m' (\alpha_i, \wt(N)) - n' (\alpha_i, \wt(M)).
\end{align*}

\noi
(i)\
By \cite[Proposition 10.1.5]{KKKO17}, we have
\begin{align*}
E_i^{(m+n)} (M \conv N) &\simeq   q^{ mn   + ( \wt(M), n\alpha_i )} M_0 \conv N_0 , \\
E_i^{(m+n)} (N \conv M) &\simeq   q^{  mn   + ( \wt(N), m\alpha_i )} N_0 \conv M_0.
\end{align*}
Using the $R$-matrix $\rmat{M,N}$, we have
\begin{align} \label{Eq: rmat M and N}
M\conv N \twoheadrightarrow M \hconv N \rightarrowtail q^{- \La(M,N)} N \conv M.
\end{align}
Note that $ E_i^{(m+n)} (M \hconv N) \ne 0$ since $\ep_i(M \hconv N) = m+n$.
By applying the exact functor $E_i^{(m+n)}$ to $\eqref{Eq: rmat M and N}$,  we have non-zero homomorphisms
\begin{align*}
q^{ mn  + ( \wt(M), n\alpha_i )} M_0 \conv N_0  \twoheadrightarrow
E_i^{(m+n)} (M \hconv N) \rightarrowtail  q^{- \La(M,N) +  mn + ( \wt(N), m\alpha_i )} N_0 \conv M_0.
\end{align*}
Proposition \ref{Prop: R-matrix properties} implies that
\[
\La(M_0, N_0) = \La(M,N) + n(\alpha_i, \wt(M)) - m(\alpha_i, \wt(N)),
\]
which gives the assertion.

\snoi
(ii) can be proved in the same manner as above.
\end{proof}

\begin{cor}[\protect{cf.\ \cite[Corollary 10.1.4]{KKKO17}}]  \label{Cor: La L(i) and M}
Let $i\in I$ and $M$ a simple module. Then we have
\begin{align*}
\La(L(i), M) &=  (\alpha_i, \alpha_i)\ep_i(M)  + (\alpha_i, \wt(M)), \\
\La(M, L(i)) &=  (\alpha_i, \alpha_i)\ep_i^*(M)  + (\alpha_i, \wt(M)).
\end{align*}
\end{cor}
\begin{proof}
Since
\begin{itemize}
\item[(i)]  $L(i)$ is real,
\item[(ii)] $\ep_i(L(i) \hconv M) = 1 + \ep_i(M)$ and $\ep_i^*( M \hconv L(i)) = 1 + \ep_i^*(M)$,
\end{itemize}
we have the assertion by applying Proposition \ref{Prop: formula for La} to $L(i)$ and $M$.
\end{proof}

\vskip 2em

\section{Determinantial modules} \label{Sec: determinatial module}

From now until Lemma \ref{Lem: M and L(i)}, we assume that $\cmA$ is arbitrary.

 Let $\Lambda \in \wlP_+$ and $ \lambda  = w\Lambda$ for $w  \in \weyl$.
Recall the cyclotomic quiver Hecke algebra $R^\La$ in \eqref{eq:cyclo}.
Take a reduced expression $\underline{w} = s_{i_1} \cdots s_{i_l}$ of $w$.
  We define
\begin{align*}
\dM(\lambda, \Lambda) \seteq F_{i_1}^{\La\,( m_1) }  \cdots F_{i_l}^{\La\,( m_l) } \trivialM,
\end{align*}
where  $m_k = \langle h_{k},   s_{i_{k+1}} \cdots s_{i_l}\Lambda \rangle$ for $k=1,\ldots l$.
Note that $ \lambda$ is an extremal weight of $V_q(\Lambda)$.
By Lemma~\ref{lem:htlt}, $R^\La(\la)$ is Morita equivalent to $\corp$,
and $\dM(\lambda, \Lambda)$ is a simple module.

By the definition, we have
\[
[\dM(\lambda, \Lambda)] \in V_q(\Lambda)_\lambda \subset \Aq[\n]_{\la-\La} ,
\] and $[\dM(\lambda, \Lambda)]$ is the unipotent quantum minor $D(\lambda, \Lambda)$ by this identification.
 Hence $\dM(\lambda, \Lambda)$ is a self-dual simple module.
Moreover the isomorphism class of $\dM(\lambda, \Lambda)$ does not depend on the choice of reduced expression of $w$.

\begin{prop}
Let $\Lambda\in \wlP_+$. For $\lambda, \mu \in \weyl \Lambda$ with $ \lambda \wle \mu$. Then there exists a self-dual simple module
$\dM(\lambda, \mu)$ such that
$$[\dM(\lambda, \mu)] = D(\lambda, \mu).$$
Moreover, its isomorphism class is unique.
\end{prop}
\Proof
The uniqueness is obvious by Theorem~\ref{Thm: categorification}.
 Let us write $\mu = u \Lambda$ with $u \in \weyl$.
We shall argue by induction on $\ell(u)$.
If $\ell(u)=0$, then it is obvious.
Hence we may assume that there exists $i\in I$ such that $u'\seteq s_iu<u$.
Set $\mu'=u'\La$.
Then, by induction hypothesis, there exists a self-dual simple module
$\dM(\lambda, \mu')$ such that
$$[\dM(\lambda, \mu')] = D(\lambda, \mu').$$
Set $n=\ang{h_i,\mu'}\in\Z_{\ge0}$.
Then we have
$e_i^{*\,(n)}D(\lambda, \mu')=D(\lambda, \mu)$ and
$e_i^{*\,(n+1)}D(\lambda, \mu')=0$.
Hence we obtain
$\eps^*_i(\dM(\lambda, \mu' ))=n$.
Moreover,
$\dM(\lambda, \mu)\seteq E_{i}^{*\, (n)}\dM(\lambda, \mu')$
is a simple module and it
satisfies
the desired condition.
\QED

One can show easily that
\eq \label{Eq: muLambda}
\gW( \dM(\lambda,\mu)) \subset \gW( \dM(\lambda,\Lambda))
\eneq
 for any $\la,\mu \in \weyl \Lambda$ with $\lambda \wle \mu $.

\begin{prop} \label{Prop: dM real}
Let $\Lambda, \Lambda' \in \wlP_+$.
\bnum
\item For $w,v \in \weyl$ with $v \le w$, we have
\[
 \dM(w\Lambda, v\Lambda) \conv \dM(w\Lambda', v\Lambda') \simeq q^{- (v\Lambda, v \Lambda' - w \Lambda')}  \dM(w( \Lambda+\Lambda'), v(\Lambda + \Lambda')) .
\]
\item  For $\lambda, \mu \in \weyl \Lambda$ with $\lambda \wle \mu$,
 $\dM(\lambda, \mu)$ is real.
\end{enumerate}
\end{prop}
\begin{proof}
The assertion follows immediately from Lemma \ref{Lem: minor}.
\end{proof}

\begin{lem}\label{lem:Wwt}
Let $M$ be an $R$-module.
Then we have
$$\gW(M)=\wt(U^+_q(\g)[M])-\wt(M).$$
 In particular, for $\Lambda\in \wlP_+$ and $\la\in \weyl\La$,
$$\gW (\dM(\la,\Lambda)) =\wt ( U^+_q(\g)u_{\la}) -\la$$
where $u_{\la}$ is an extremal vector of weight $\la$ in $V_q(\Lambda)$.
\end{lem}
\begin{proof}
The assertion follows from the equivalence of the following statements:
\begin{enumerate}
\item[(i)] $\gamma \in \gW(M)$,
\item[(ii)] $E_{i_n}\cdots E_{i_1}[M]\not=0 \text{\ for some $(i_1,\ldots,i_n)\in I^\gamma$}$,
\item[(iii)] $ \left( U^+_q(\g)[M] \right) _{\wt(M)+\gamma}\not=0.$
\end{enumerate}
\end{proof}

\begin{lem} \label{Lem: weight} Let $\La\in\wlP_+$.
Let $\lambda, \mu \in \weyl \Lambda$ with $ \lambda  \wle \mu $.
Let $u_\la\in V(\La)$ be the extremal weight vector of weight $\la$.
\bnum
\item
If $\beta\in \gW( \dM( \lambda, \mu ) )$, then
\[
\lambda+\beta \in\wt ( U^+_q(\g)u_{\la}) \subset \wt(  V(\Lambda)).
\]
\item
If $\gamma \in \sgW( \dM( \lambda, \mu ) )$, then
\[
\mu - \gamma \in\wt ( U^+_q(\g)u_{\la}) \subset\wt(  V(\Lambda)).
\]
\end{enumerate}
\end{lem}
\begin{proof}
(i) If $\beta\in \gW( \dM( \lambda, \mu ) )$, then
we have
$\beta\in\gW\bl M(\la,\La)\br$ by \eqref{Eq: muLambda}.
Then, Lemma~\ref{lem:Wwt} implies $\la+\beta\in\wt\bl U^+_q(\g)u_{\la}\br$.

\noi
(ii)\
If $\gamma \in \sgW( \dM( \lambda, \mu ) )$,
then we have $\mu-\la-\gamma\in \gW( \dM( \lambda, \mu ) )$.
Hence (i) implies that
$$\wt\bl U^+_q(\g)u_{\la}\br\ni \la+(\mu-\la-\gamma)=\mu-\gamma.$$
\end{proof}

\begin{thm} \label{Thm: M and C}
Let $\Lambda \in \wlP_+$ and let $w, v \in \weyl $.
\bnum
\item
Assume that $w', v'\in \weyl $ satisfy
\begin{align*}
w \ge w' \ge v' \ge v, \quad  \ell(w) = \ell(w') + \ell(w'^{-1}w), \quad
 \ell(v') = \ell(v) + \ell(v^{-1}v').
\end{align*}
Then we have $\catC_{w',v'}\subset \catC_{w,v}$.
\item
We have
\[
\dM( w\Lambda, v\Lambda) \in \catC_{w,v}.
\]
\ee
\end{thm}

\begin{proof}
(i) follows from
$\prD \cap w' \nrD \subset \prD \cap w \nrD$ and
$\prD \cap v' \prD \subset \prD \cap v \prD$.

\snoi
(ii)\ Let  $M = \dM( {w}\Lambda, {v}\Lambda)$.
We take $\beta \in \gW(M)$.
Then, Lemma \ref{Lem: weight} (i) tells that
\begin{align*}
{w}\Lambda + \beta \in \wt(V(\Lambda)).
\end{align*}
Since $ \Lambda + {w}^{-1}\beta \in \wt(V(\Lambda)) $ and $\Lambda$ is the highest weight of $V(\Lambda)$,
we have ${w}^{-1}\beta \in \rlQ_-$, which implies that
$\beta \in w\rlQ_-$.
Thus we have
$M \in \catC_w$.

We now take $\gamma \in \sgW(M)$.  It follows from Lemma \ref{Lem: weight} (ii) that
\begin{align*}
{v}\Lambda - \gamma \in \wt(V(\Lambda)).
\end{align*}
As $\Lambda - {v}^{-1}\gamma \in \wt(V(\Lambda))$, we have $v^{-1}\gamma \in \rlQ_+$.
Thus we have
$\gamma \in v \rlQ_+$,
which implies $M \in \catC_{*,v}$.
\end{proof}

The following proposition is proved in \cite[Theorem 10.3.1]{KKKO17}
when the quiver Hecke algebra is symmetric.

\Prop \label{prop:Det2}
Let $\Lambda \in \wlP_+$ and
assume that $\la,\la',\la''\in \weyl \La$ satisfy $\la\wle\la'\wle \la''$.
Then there is an epimorphism
$$\dM(\la,\la')\conv \dM(\la',\la'')\epito
\dM(\la,\la'').$$
\enprop
\Proof
Assume that
$w, x,y,z\in\weyl$ such that $z\le y\le x\le w$.
Then Theorem~\ref{Thm: M and C} implies that
\eqn
&&\sgW\bl\dM(x\La,y\La)\br\subset\rtl_+\cap y\rtl_+,
\ \gW\bl\dM(y\La,z\La)\br\subset\rtl_+\cap y\rtl_-,\\
&&\sgW\bl\dM(y\La,z\La)\br\subset\rtl_+\cap  z\rtl_+ ,
\ \gW\bl\dM(z\La,w\La)\br\subset\rtl_+\cap z\rtl_-.\eneqn
It implies that $\bl \dM(x\La,y\La),\dM(y\La,z\La)\br$ and
$\bl \dM(x\La,y\La),\dM(y\La,z\La),\dM(z\La,w\La)\br$ are unmixed.
Thus we conclude that
$$\text{$\dM(x\La,y\La)\conv\dM(y\La,z\La)$ and
$\dM(x\La,y\La)\conv\dM(y\La,z\La)\conv\dM(z\La,w\La)$
have simple heads.}$$

\noi
(i) {\em Case $\la''=\La$.}\quad Set
$\la'=y\La$.
We shall prove the assertion by induction on $\ell(y)$.
When $y=1$ it is obvious.
Otherwise take $i\in I$ such that $y'\seteq s_iy<y$.
Then
$$\dM(\la,y\La)\conv\dM(y\La,y'\La)\conv\dM(y'\La,\La)$$
has a simple head.
By the definition we have an epimorphism
$\dM(\la,y\La)\conv\dM(y\La,y'\La)\epito \dM(\la,y'\La)$
since $\dM(y\La,y'\La)\simeq L(i^n)$ with $n=\ang{h_i,y'\La}$
and $\dM(\la,y\La)\simeq E_i^{*\,(n)}\dM(\la,y'\La)$.
By induction hypothesis there exists an epimorphism
$\dM(\la,y'\La)\conv \dM(y'\La,\La)\epito \dM(\la,\La)$.
Thus we have an epimorphism
$\dM(\la,y\La)\conv\dM(y\La,y'\La)\conv\dM(y'\La,\La)\epito \dM(\la,\La)$.
Hence
$ \dM(\la,\La)$ is the simple
head of $\dM(\la,y\La)\conv\dM(y\La,y'\La)\conv\dM(y'\La,\La)$.
On the other hand we have an epimorphism
$$\dM(\la,y\La)\conv\dM(y\La,y'\La)\conv\dM(y'\La,\La)
\epito \dM(\la,y\La)\conv \dM(y\La,\La).$$
Hence $\dM(\la,\La)$ is equal to the simple head of
$\dM(\la,y\La)\conv \dM(y\La,\La)$.

\medskip\noi
(ii)\ {\em Reduction to the case $\la''=\La$.}\quad
Assume that $\la''\neq\La$.
Then there is $i\in I$ such that $n\seteq-\ang{h_i,\la''}>0$.
By induction on $\Ht(\La-\la'')$, we may assume that the assertion holds for $s_i\la''$.
Therefore, we have a homomorphism
\eqn
\dM(\la,\la')\conv\dM(\la',\la'')
&\simeq&
\dM(\la,\la')\conv\bl E^*_i{}^{(n)}\dM(\la',s_i\la'')\br\\
&\to&
E^*_i{}^{(n)}\bl\dM(\la,\la')
\conv\dM(\la's_i\la'')\br\\
&\epito&
E^*_i{}^{(n)}\dM(\la,s_i\la'')\\
&\simeq&
\dM(\la,\la'').
\eneqn
It is obvious that the composition does not vanish, and hence
it is an epimorphism.
\QED

\begin{cor} Let $\Lambda \in \wlP_+$.
\bnum
\item
Let $i\in I$ and $v \in \weyl$ such that $vs_i<v$.
If $n \seteq \langle h_i,\Lambda\rangle>0$, then
$\dM(v\Lambda, vs_i\Lambda)$ is a $\wupreceq{v}$-cuspidal $R(-n v \alpha_i)$-module, where $\wupreceq{v}$ is a convex order
corresponding to a reduced expression $\underline{v}$ given in
{\rm Proposition \ref{Prop: convex preorder for w}}.

\item
We take a reduced expression $\underline{w} = s_{i_1} \cdots s_{i_\ell}$ of $w\in \weyl$.
Then the sequence obtained from
$
 ( \dM(w_\ell \Lambda,w_{\ell-1}\Lambda),\ldots, \dM(w_1\Lambda,w_0\Lambda) )
$
by removing $\dM(w_k\Lambda,w_{k-1}\Lambda)$ such that $w_k\Lambda=w_{k-1}\Lambda$,
is a $\wupreceq{w}$-cuspidal decomposition of $\dM(w\Lambda,\Lambda)$, where
$w_k=s_{i_1}\cdots s_{i_k}$ for $k=1,\ldots, \ell$.
\end{enumerate}

\end{cor}
\begin{proof}
(i) \ Set $v'=vs_i$. Then $-\wt\bl\dM(v\Lambda, vs_i\Lambda)\br=nv'\al_i$
and Theorem~\ref{Thm: M and C} implies that
$\gW\bl \dM(v\Lambda, vs_i\Lambda)\br\subset \Sp(\prD \cap v\nrD)
=\Sp(\set{\beta\in\prD}{\beta\wupreceq{v}v'\al_i})$.
Hence
$\dM(v\Lambda, vs_i\Lambda)$ is a $\wupreceq{v}$-cuspidal $R(n v' \alpha_i)$-module.

\snoi
(ii)\ By Proposition~\ref{prop:Det2}, there exists an epimorphism
$$\dM(w_\ell \Lambda,w_{\ell-1}\Lambda)\conv\cdots\conv\dM(w_1\Lambda,w_0\Lambda)\epito \dM(w\Lambda,\La).$$
Hence (ii) follows from (i).
\end{proof}

Let  $w, v \in \weyl$ with $v \le w$ and fix a reduced expression $\underline{w} = s_{i_1}s_{i_2} \cdots s_{i_l} $ of $w$.
For $k=1, \ldots, l$, we set
\begin{align} \label{Eq: def of wk}
w_{\le k} \seteq s_{i_1} \cdots s_{i_k}, \qquad  w_{\ge k+1} \seteq w_{\le k}^{-1} w.
\end{align}

We shall define $v_{\le k}$ and $v_{ \ge k }$ for $k=1, \ldots, l$ by
\begin{equation}  \label{Eq: def of wk vk}
\begin{aligned}
&{\rm(i)\ }  v_{\ge k} = (v_{\le k-1})^{-1} v , \\
&{\rm(ii)\ }  v_{\le k} =
\bc v_{\le k-1}s_{i_k} &\text{if $ s_{i_k} v_{\ge k} < v_{\ge  k}$,} \\
 v_{\le k-1}&\text{if $ s_{i_k} v_{\ge k} > v_{\ge  k}$.}\ec
\end{aligned}
\end{equation}
Here we set $w_{\le 0} = v_{\le 0} = \mathrm{id} \in \weyl$.
Note that $w_{\le l}=w$, $v_{\le l}=v$, $ v_{\le k} \le w_{\le k} $, $ v_{\ge k} \le w_{\ge k} $ and  $ \ell(v) = \ell(v_{\le k-1}) + \ell(v_{ \ge k}) $ for $k=1, \ldots, l$.

\Prop \label{prop: M(wk,vk) and Cwv}
Let $\Lambda \in \wlP_+$ and $w,v \in \weyl$ with $v \le w$.
We fix a reduced expression
$\underline{w} = s_{i_1} \cdots s_{i_l}$.
For $k =0, 1, \ldots, l$,
\[
\dM(w_{\le k}\Lambda, v_{\le k}\Lambda) \in \catC_{w,v}
\]
where $ w_{\le k}$ and $v_{\le k}$ are defined in $\eqref{Eq: def of wk}$ and $\eqref{Eq: def of wk vk}$.
\enprop
\Proof
Note that $\dM(w_{\le k}\Lambda, v_{\le k}\Lambda) \in \catC_{w_{ \le k }}\subset\catC_w$
follows from Theorem~\ref{Thm: M and C}.
Hence it remains to prove
$\sgW\bl\dM(w_{\le k}\Lambda, v_{\le k}\Lambda)\br \in v\rtl_+$.

Assume that
$\beta\in \sgW\bl\dM(w_{\le k}\Lambda, v_{\le k}\Lambda)\br$.
Then Lemma~\ref{Lem: weight} (ii) implies that
\eq
v_{\le k}\La-\beta\in \wt\bl U^+_q(\g)u_{w_{\le k}\La}\br.
\label{cond:wv}
\eneq
Hence it is enough to show that
$$\text{if $\beta\in\rtl$ satisfies condition \eqref{cond:wv},
then we have $\beta\in v \rtl_+$.}$$
We shall prove it by induction on $\ell(w)+\ell(v)$.

\snoi
Since it is obvious when $k=0$, let us assume that $k>0$.
Set $i=i_1$.

\snoi
(1)\ {Case $s_iv>v$.}
We have $\al_i\in v\prD$.
Set $w'\seteq s_iw<w$ and $v'\seteq v$.
We define $w'_{\le k}$ and $v'_{\le k}$ with respect
to the reduced expression $w'=s_{i_2}\cdots s_{i_l}$ and $v'$.
Then we have
$w_{\le k}=s_iw'_{\le k-1}$ and $v_{\le k}=v'_{\le k-1}$.
Hence
$v'_{\le k-1}\La-\beta\in \wt\bl U^+_q(\g)u_{s_iw'_{\le k-1}\La}\br
\subset  \wt\bl U^+_q(\g)u_{w'_{\le k-1}\La}\br-\Z_{\ge0}\al_i$.
Hence for some $m\in\Z_{\ge0}$ we have
$v'_{\le k-1}\La-(\beta-m\al_i)\in \wt\bl U^+_q(\g)u_{w'_{\le k-1}\La}\br$.
Then, induction hypothesis implies that
$\beta-m\al_i\in v'\rtl_+=v\rtl_+$.
Hence we obtain $\beta\in v\rtl_++m\al_i\subset v\rtl_+$.

\snoi
(2)\ {Case $s_iv<v$.}
Set $w'\seteq w$ and $v'\seteq s_iv<v$.
We define $w'_{\le k}$ and $v'_{\le k}$ with respect
to the reduced expression $w'=s_{i_1}\cdots s_{i_l}$ and $v'$.
Then we have
$w_{\le k}=w'_{\le k}$ and $v_{\le k}=s_iv'_{\le k}$.
Hence
$s_iv'_{\le k}\La-\beta\in \wt\bl U^+_q(\g)u_{w'_{\le k}\La}\br$.

Since $s_iw'_{\le k}<w'_{\le k}$,
the set $\wt\bl U^+_q(\g)u_{w'_{\le k}\La}\br$
is invariant by the action of $s_i\in\weyl$.
Hence we have
$$v'_{\le k}\La-s_i\beta\in \wt\bl U^+_q(\g)u_{w'_{\le k}\La}\br.$$
Then induction hypothesis implies that
$s_i\beta\in v'\rtl_+$.
Hence we obtain
$\beta\in s_iv'\rtl_+=v\rtl_+$.
\QED

We assume that $R(\beta)$ is symmetric in  the rest of this section.

\begin{lem} \label{Lem: M and L(i)}
Let $i\in I$ and let $\lambda, \mu \in \weyl \Lambda $ with $ \lambda\wle\mu$ for $\Lambda \in \wlP_+$.
If $ \langle h_i, \lambda \rangle \le 0 $ and $ \langle h_i, \mu \rangle \ge 0 $, then
$\dM(\lambda, \mu)$ \scts with $L(i)$.
\end{lem}
\begin{proof}
By the construction, we have
\begin{align*}
\ep_i(\dM(\lambda, \mu)) = - \langle h_i, \lambda \rangle, \quad
\ep^*_i(\dM(\lambda, \mu)) =  \langle h_i, \mu \rangle, \quad
 \wt( \dM(\lambda, \mu) ) = \lambda - \mu.
\end{align*}
By Corollary \ref{Cor: La L(i) and M}, we have
\begin{align*}
\La& (L(i), \dM(\lambda, \mu)) +  \La(\dM(\lambda, \mu), L(i)) \\
&= (\alpha_i, \alpha_i)\ep_i(  \dM(\lambda, \mu) ) + ( \alpha_i, \wt( \dM(\lambda, \mu) ) ) +
(\alpha_i, \alpha_i) \ep^*_i(  \dM(\lambda, \mu) ) + ( \alpha_i, \wt( \dM(\lambda, \mu) ) ) \\
&=0,
\end{align*}
which yields the assertion by Proposition \ref{Prop: R-matrix properties}.
\end{proof}

\begin{thm} \label{Thm: commutation for dM}
Let $\Lambda, \Lambda' \in \wlP_+$ and $w,v \in \weyl$ with $v \le w$.
We fix a reduced expression
$\underline{w} = s_{i_1} \cdots s_{i_l}$.
For $1\le k\le j\le l$,
\[
\text{$\dM(w_{\le j}\Lambda, v_{\le j}\Lambda)$ and $\dM(w_{\le k}\Lambda', v_{\le k}\Lambda')$ \sct,}
\]
where $ w_{\le k}$ and $v_{\le k}$ are defined in $\eqref{Eq: def of wk}$ and $\eqref{Eq: def of wk vk}$.
\end{thm}
\begin{proof}
By replacing $w$ and $v$ with $w_{\le j}$ and $v_{\le j}$,
it is enough to show that
$\dM(w\Lambda, v\Lambda)$ and $\dM(w_{\le k}\Lambda', v_{\le k}\Lambda')$ \sct.
We shall show it by induction on $l$.

Let $i = i_1$.  We assume first that $v_{\le 1} = \mathrm{id}$. Then $s_i v > v$ and $s_i v_{\le k} > v_{\le k}$ for any $k$. As $\langle h_i, w\Lambda \rangle \le 0 $ and $\langle h_i, v\Lambda \rangle \ge 0$,
$\dM(w\Lambda, v\Lambda)$ and $\dM(w_{\le k}\Lambda', v_{\le k}\Lambda')$ \sct with $L(i)$ by Lemma \ref{Lem: M and L(i)}.
As
\begin{align*}
\tEm_i  \dM( w\Lambda, v\Lambda) &\simeq \dM(s_i w\Lambda, v\Lambda), \\
\tEm_i  \dM( w_{\le k} \Lambda', v_{\le k}\Lambda') &\simeq \dM(s_i w_{\le k} \Lambda', v_{\le k}\Lambda'),
\end{align*}
by the induction hypothesis on $l$, $\tEm_i  \dM( w\Lambda, v\Lambda)$ and $\tEm_i  \dM( w_{\le k} \Lambda', v_{\le k}\Lambda')$
\sct.
Hence $\dM(w\Lambda, v\Lambda)$ and $\dM(w_{\le k}\Lambda', v_{\le k}\Lambda')$ \sct by Lemma \ref{Lem: com M N}.

Suppose $v_{\le 1} = s_{i}$. Letting $v' = s_iv$, we have $v'_{\le k} =   s_i v_{\le k}$ for $k=1, \ldots, l$. Since $s_{i} v' > v'$, it follows from the above that
$\dM(w\Lambda, v'\Lambda)$ and $\dM(w_{\le k}\Lambda', v'_{\le k}\Lambda')$ \sct.
Therefore, Lemma \ref{Lem: com tE tEs} implies that
$\dM(w\Lambda, v\Lambda)$ and $\dM(w_{\le k}\Lambda', v_{\le k}\Lambda')$ \sct.
\end{proof}

\begin{lem} \label{Lem: La}
Let $i\in I$, $\lambda, \mu \in \weyl \Lambda$ and $\lambda', \mu' \in \weyl \Lambda'$ such that
\[
\langle h_i, \lambda' \rangle \le 0, \quad \langle h_i, \mu \rangle \ge 0, \quad s_i \lambda \wle  \mu, \quad s_i \lambda'   \wle  \mu'.
\]
\bnum
\item If $ \langle h_i, \lambda \rangle \le 0 $, then
 we have
\[
\La(\dM(\lambda, \mu), \dM( \lambda', \mu') )
-\La(\dM(s_i\lambda, \mu), \dM( s_i\lambda', \mu') )
= (\lambda + \mu,  \mu' - \lambda')-(s_i\lambda + \mu, \mu' - s_i \lambda') .
\]
\item If $ \langle h_i, \mu' \rangle \ge 0 $, then
  we have
$$\La(\dM(\lambda, \mu), \dM( \lambda', \mu') )
-\La(\dM(\lambda, s_i\mu), \dM( \lambda', s_i\mu') )
= (\lambda + \mu,  \mu' - \lambda' )
-(\lambda + s_i\mu,  s_i\mu' - \lambda').$$
\end{enumerate}
\end{lem}
\begin{proof}
We set
\[
\La_{\lambda, \mu, \lambda', \mu' } \seteq \La(\dM(\lambda, \mu), \dM( \lambda', \mu') ).
\]

\noi
(i) By Lemma \ref{Lem: M and L(i)} and Proposition \ref{Prop: dM real}, $\dM(\lambda, \mu)$ is real and \scts with $L(i)$.
 Let $m = - \langle h_i, \lambda \rangle$ and $n = -\langle h_i, \lambda' \rangle$. It follows from Lemma \ref{Lem: comm with L(i)}  and  Proposition \ref{Prop: formula for La} that
\begin{align*}
\La_{\lambda, \mu, \lambda', \mu' } -
\La_{s_i\lambda, \mu, s_i\lambda', \mu' }&= - n(\alpha_i, \wt(\dM(\lambda, \mu))) + m(\alpha_i, \wt(\dM(\lambda', \mu'))) \\
&= 
 - n(\alpha_i, \la-\mu)   + m(\alpha_i, \la'-\mu') \\
\end{align*}
and
\eqn
&&(\lambda + \mu,  \mu' - \lambda')-(s_i\lambda + \mu, \mu' - s_i \lambda')\\
&&\hs{10ex}
=(\lambda + \mu,  \mu' - \lambda')
-(\lambda + \mu+m\al_i, \mu' - \lambda'-n\al_i)\\
&&\hs{10ex}=-m(\al_i,\mu' - \lambda')+n(\al_i,\lambda + \mu)+mn(\al_i,\al_i)\\
&&\hs{10ex}=m(\al_i,\lambda'-\mu')-n(\al_i,\lambda - \mu)
+2n(  \al_i  ,\la)+mn(\al_i,\al_i)\\
&&\hs{10ex}=m(\al_i,\lambda'-\mu')-n(\al_i,\lambda - \mu).
\eneqn
\noi
(ii) Note that $\dM(\lambda', \mu')$ is real and \scts with $L(i)$ by Lemma \ref{Lem: M and L(i)} and Proposition \ref{Prop: dM real}.
Let $p =  \langle h_i, \mu \rangle$ and $q = \langle h_i, \mu' \rangle$. By Lemma \ref{Lem: comm with L(i)} and Proposition \ref{Prop: formula for La}, we obtain
\begin{align*}
\La_{\lambda, \mu, \lambda', \mu' }-\La_{\lambda, s_i\mu, \lambda', s_i\mu' } &=  q(\alpha_i, \wt(\dM(\lambda, \mu))) - p(\alpha_i, \wt(\dM(\lambda', \mu'))) \\
&= q(\alpha_i, \lambda  - \mu) - p(\alpha_i, \lambda' -\mu' ).
\end{align*}
and
\eqn
&&(\lambda + \mu,  \mu' - \lambda' )
-(\lambda + s_i\mu,  s_i\mu' - \lambda')\\
&&\hs{10ex}=(\lambda + \mu,  \mu' - \lambda' )
-(\lambda+\mu-p\al_i,  \mu' - \lambda'-q  \al_i )\\
&&\hs{10ex}=q(\lambda+\mu,\al_i)
+p(\al_i,\mu' - \lambda')-pq(\al_i,\al_i)\\
&&\hs{10ex}=q(\lambda-\mu,\al_i)
-p(\al_i,\lambda'-\mu' )+2q(\al_i,\mu)-pq(\al_i,\al_i)\\
&&\hs{10ex}=q(\lambda-\mu,\al_i)
-p(\al_i,\lambda'-\mu').
\eneqn
\end{proof}

\begin{thm} \label{Thm: degree of R}
Let $\Lambda, \Lambda' \in \wlP_+$ and $w,v\in \weyl$  with $ v \le w$.
We fix a reduced expression $\underline{w} = s_{i_1} \cdots s_{i_l}$.
For $k=0,1, \ldots, l$,
we have
\begin{align*}
\La( \dM(w\Lambda, v\Lambda),  \dM(w_{\le k}\Lambda', v_{\le k}\Lambda')) = ( w \Lambda + v\Lambda, v_{\le k}\Lambda' - w_{\le k}\Lambda' ).
\end{align*}
\end{thm}
\begin{proof}
We use induction on $l \seteq \ell(w) + \ell(v)$.
As it is obvious that the assertion holds if either $l=0$ or $k=0$, we assume that $l >0$ and $k > 0$.
Let $i = i_1$. Then we have
\[
\langle  h_i, w\Lambda \rangle \le 0, \quad \langle  h_i, w_{\le k}\Lambda \rangle \le 0.
\]

Suppose that $s_i v > v$. Note that $ \langle  h_i, v\Lambda' \rangle \ge 0 $.
Let $\tilde{w} = s_i w$ and $\tilde{v} = v$. Since $ \tilde{w}_{\le k-1} = s_i w_{\le k} $ and $ \tilde{v}_{\le k-1} = v_{\le k}$,
the induction hypothesis implies
\begin{align*}
\La( \dM(s_iw\Lambda, v \Lambda),   \dM(s_iw_{\le k}\Lambda', v_{\le k} \Lambda') )
& = \La( \dM(\tilde{w}\Lambda, \tilde{v} \Lambda),   \dM( \tilde{w}_{\le k-1}\Lambda', \tilde{v}_{\le k-1} \Lambda') ) \\
&= ( s_iw\Lambda + v \Lambda,  v_{\le k} \Lambda' - s_iw_{\le k}\Lambda' ).
\end{align*}
By Lemma \ref{Lem: La},  we have
\[
\La( \dM(w\Lambda, v\Lambda),  \dM(w_{\le k}\Lambda', v_{\le k}\Lambda')) = ( w \Lambda + v\Lambda, v_{\le k}\Lambda' - w_{\le k}\Lambda' ).
\]

We now suppose that $ s_iv < v$. Note that $ \langle  h_i, s_iv\Lambda' \rangle \ge 0 $.
Let $\hat{w} =  w$ and $\hat{v} = s_iv$. Then $ \hat{w}_{\le k} = w_{\le k} $ and $ \hat{v}_{\le k} = s_iv_{\le k}$, so we have
\begin{align*}
\La( \dM(w\Lambda, s_iv \Lambda), \dM(w_{\le k}\Lambda', s_iv_{\le k} \Lambda') )
 &= \La( \dM(\hat{w}\Lambda, \hat{v} \Lambda), \dM(\hat{w}_{\le k}\Lambda', \hat{v}_{\le k} \Lambda') ) \\
&= (  w\Lambda + s_iv \Lambda,  s_iv_{\le k} \Lambda' - w_{\le k}\Lambda' )
\end{align*}
by the induction hypothesis. Therefore, Lemma \ref{Lem: La} implies that
\[
\La( \dM(w\Lambda, v\Lambda),  \dM(w_{\le k}\Lambda', v_{\le k}\Lambda')) = ( w \Lambda + v\Lambda, v_{\le k}\Lambda' - w_{\le k}\Lambda' ).
\]
\end{proof}

\section{Finite ADE types} \label{Sec: ADE}

In this section, we assume that $\cmA$ is of finite $ADE$ type.
Then any corresponding quiver Hecke algebra $R$
is isomorphic to a symmetric quiver Hecke algebra.
We assume further that the base field $\corp$ is of characteristic $0$.
Under this assumption, the set of the isomorphism classes of simple $R$-modules corresponds to the upper global basis of $\Aq[\n]$  (\cite{R11, VV09}).

For $i\in I$, we define $\iR \gmod$ (resp.\ $\Ri \gmod$) to be the full subcategory of $R\gmod$ whose objects $M$ satisfy $\ep_i(M)=0$ (resp.\ $\ep^*_i(M)=0$).
Note that the categories $\iR \gmod$ and $\Ri \gmod$ are stable under
 taking  subquotients, extensions, convolution products and grading shifts.
It was proved in \cite[Proposition 3.5]{Kato14} that there exist reflection functors
\begin{align*}
\rF_i &: \Ri (\beta) \gmod \buildrel \sim \over \longrightarrow  \iR (s_i\beta) \gmod, \\
\rFs_i &: \iR (\beta) \gmod \buildrel \sim \over \longrightarrow  \Ri (s_i\beta) \gmod,
\end{align*}
which  are equivalences of categories
commuting with the grading shift functor. 
Note that, by \cite[Proposition 3.5]{Kato14}
\begin{align} \label{Eq: adjoint}
\text{$\rF_i$ and $\rFs_i$ are quasi-inverse.}
\end{align}
The functors $\rF_i$ and $\rFs_i$ are counterparts of the {\it Saito crystal reflections} on the crystals \cite{Saito94}, which are defined by
\begin{align*}
\cR_i &: \Bi  \longrightarrow  \iB ,  \qquad b \mapsto {\tilde{f}_i}^{* \hskip 0.1em  \ph_i(b)} {\te_i}^{\hskip 0.1em \ep_i(b)}b, \\
\cRs_i &: \iB  \longrightarrow  \Bi, \qquad b \mapsto {\tilde{f}_i}^{  \ph_i^*(b)} {\te_i}^{\hskip 0.1em * \hskip 0.1em \ep_i^*(b)}b,
\end{align*}
where $\iB = \{ b \in B(\infty) \mid \ep_i(b)=0  \}$ and $\Bi = \{ b \in B(\infty) \mid \ep^*_i(b)=0  \}$.
Letting $L(b)$ be the self-dual simple $R$-module corresponding to $b \in B(\infty)$, it was shown in \cite[Theorem 3.6]{Kato14} that
\begin{align} \label{Eq: refl and crystal}
\rF_i(L(b)) \simeq L(\cR_i b) \quad \text{ and } \quad \rFs_i(L(b')) \simeq L(\cRs_i b')
\end{align}
for $b \in \Bi$ and $b' \in \iB$.

For a reduced expression $\underline{w}$ of $w\in W$, let us recall that $\wupreceq{w}$ denotes a convex order corresponding to $\underline{w}$
defined in Proposition \ref{Prop: convex preorder for w}.
We consider a reduced expression $\underline{w_0} = s_{i_1}\cdots s_{i_\ell}$ of the longest element $w_0 \in \weyl$.
If we set $L_k$ to be the $\wupreceq{w_0} $-cuspidal module corresponding to
the positive root $\beta_k = s_{i_1}\cdots s_{i_{k-1}}(\alpha_{i_k})$, we have
\begin{align} \label{Eq: cusp and refl}
L_k \simeq L( \cRs_{i_1} \cRs_{i_2} \cdots \cRs_{i_{k-1}}(f_{i_k})  ) \simeq \rFs_{i_1} \rFs_{i_2} \cdots \rFs_{i_{k-1}} L( i_k  )
\end{align}
Further, it was shown in \cite[Lemma 4.2]{Kato14} that,
for $1 \le p < q \le \ell$ and $(t_p, \ldots, t_q) \in \Z_{\ge0}^{q-p+1}$,
\begin{equation} \label{Eq: monoidal}
\begin{aligned}
&
\rFs_{p,q-1}(L(i_{q}^{t_q})) \conv \rFs_{p,q-2}(L(i_{q-1}^{t_{q-1}})) \conv \cdots \conv  \rFs_{p,p}(L(i_{p+1}^{t_p}))  \\
&\quad  \simeq
\rFs_{i_p} \Big(
\rFs_{p+1,q-1}(L(i_{q}^{t_q})) \conv \rFs_{p+1,q-2}(L(i_{q-1}^{t_{q-1}})) \conv \cdots \conv L(i_{p+1}^{t_p})
\Big),
\end{aligned}
\end{equation}
where $ \rFs_{a,b} \seteq \rFs_{i_a} \rFs_{i_{a+1}} \cdots \rFs_{i_b}  $ for $1 \le a \le b \le \ell$.

\begin{lem} \label{Lem: Ti and C}
Let $i \in I$, $w \in \weyl$, and let $M$ be a simple $R$-module.
\bnum
\item Suppose that $M \in \catC_{w}$. Then we have
\bna
\item if $s_iw < w$ and $\ep_i^*(M) = 0$, then $ \rF_i(M) \in \catC_{s_iw}$,
\item if $s_iw > w$, then $ \rF_i^*(M) \in \catC_{s_iw}$.
\end{enumerate}
\item Suppose that  $M \in \catC_{*,w}$. Then we have
\bna
\item if $s_i w < w$, then $ \rF_i(M) \in \catC_{*, s_i w}$,
\item if $s_i w > w$ and $\ep_i(M) = 0$, then $ \rF_i^*(M) \in \catC_{*, s_i w}$.
\end{enumerate}
\end{enumerate}
\end{lem}
\begin{proof}
Let $w'=w^{-1}w_0 \in  \weyl$
so that $w w' = w_0$ with $\ell(w_0) = \ell(w) + \ell(w')$.
Let us fix reduced expressions $\underline{w} = s_{i_1}s_{i_2}\cdots s_{i_p}$ and $\underline{w'} = s_{j_1}s_{j_2}\cdots s_{j_q}$. We set
$\underline{w}^- \seteq s_{i_2}\cdots s_{i_p}$ and $\underline{w_0} \seteq \underline{w} \hskip 0.1em  \underline{w'}$.
 We write the $\wupreceq{w_0}$-decomposition of $M$ as follows:
$$
\cd_{\preceq^{\underline{w_0}}}(M) = (M_1, M_2, \ldots, M_h).
$$

\noi
(i) We assume that $M \in \catC_{w}$. Note that $-\wt(M_k) \in  \Z_{>0} (  \prD \cap w \nrD )$ for any $k$ by Proposition \ref{Prop: membership}.

\snoi
(i\,a)\ Suppose that $s_iw < w$ and $\ep_i^*(M) = 0$.
Without loss of generality, we assume that $i = i_1$.
 Then $\alpha_i$ is the smallest element with respect to $\wupreceq{w_0}$.
 The assumption $\ep_i^*(M) = 0$
implies that $\al_i\wuprec{w_0}-\wt(M_h)$,
which implies $\al_i\wuprec{w_0}-\wt(M_k)$ for $1\le k \le h$.
Hence $\al_i\not\in\sgW(M_k)$ and $M_k \in R_i \gmod$.
By $\eqref{Eq: adjoint}$ and $\eqref{Eq: cusp and refl}$, $\rF_i(M_k)$ are cuspidal modules with respect to the convex order $\wpreceq{ \underline{w}^- }$ satisfying $-\wt(\rF_i(M_k)) \wsucceq{ \underline{w}^- } -\wt(\rF_i(M_{k+1}))$ for $1 \le k \le h-1$.
It follows from $\eqref{Eq: adjoint}$ and $\eqref{Eq: monoidal}$ that $\cd_{\wpreceq{ \underline{w}^-}}(\rF_i(M))$ is $(\rF_i(M_1), \rF_i(M_2), \ldots, \rF_i(M_h))$.
Thus
$\rF_i(M)$ is contained in $\catC_{s_iw}$ by Proposition \ref{Prop: membership}.

\snoi
(i\,b)\ If $s_iw > w$, then $\ep_i(M)=0$ by Lemma \ref{Lem: ep and ep*} and  we have $\cd_{\wpreceq{s_i \underline{w}}}(\rFs_i(M)) = (\rFs_i(M_1), \rFs_i(M_2), \ldots, \rFs_i(M_h))$ by $\eqref{Eq: cusp and refl}$ and $\eqref{Eq: monoidal}$.
Thus $\rFs_i(M) \in \catC_{s_iw}$ by Proposition \ref{Prop: membership}.

\snoi
(ii) We now assume that $M \in \catC_{*,w}$. Note that $-\wt(M_k) \in  \Z_{>0}(  \prD \cap w \prD)$ for any $k$ by Proposition \ref{Prop: membership}.

\noi
(ii\,a) Suppose that $s_iw < w$.  Then $\ep_i^*(M)=0$ by Lemma \ref{Lem: ep and ep*}. Without loss of generality, we assume that $i = i_1$. Then $- \wt(M_k) \succ \alpha_i$ so that $\ep_i^*(M_k)=0$ for $k=1, \ldots, h$.
By $\eqref{Eq: adjoint}$ and $\eqref{Eq: monoidal}$, we have
$\cd_{\wpreceq{ \underline{w}^-}}(\rF_i(M))=(\rF_i(M_1), \rF_i(M_2), \ldots, \rF_i(M_h))$.
This implies that $\rF_i(M)$ is contained in $\catC_{ *, s_iw}$ by Proposition \ref{Prop: membership}.

\noi
(ii\,b)\ Suppose that $s_iw > w$ and $\ep_i(M) = 0$. Let $u=w^{-1}s_iw_0 \in \weyl$ so that $w u = s_i w_0$ and $\ell(w u)$ = $\ell(w)$ + $\ell(u)$.
Fix a reduced expression $\underline{u}$ of $u$, and let $s_j$ be the simple reflection with $s_iw_0 = w_0s_j$.
We set $\preceq' \seteq \wpreceq{\underline{w} \hskip 0.1em  \underline{u} s_j} $.
By the construction, $\alpha_i$ is largest among $\prD$ with respect to the convex order $\preceq'$.
Let $(M_1' , \ldots, M_r')$ be the $\preceq'$-cuspidal decomposition of $M$. By Proposition \ref{Prop: membership}, we know
\[
- \wt(M_k') \succ' \beta \qquad \text{ for $k=1, \ldots, r$ and $\beta \in \prD \cap w \nrD$.}
\]
Since $\ep_i(M)=0$, we have $-\wt(M_1') \notin \Z_{>0} \alpha_i$.
Note that $\preceq'=\wpreceq{\underline{w} \hskip 0.1em \underline{u}}$ as a convex order on $\Delta_+$
since $\alpha_i$ is largest with respect to $\wpreceq{\underline{w} \hskip 0.1em \underline{u}}$, and $M_k' \in \catC_{wu}$ for $1\le k\le h$.
Thus, applying $\rFs_i$ to $M$, we have that
$(\rFs_i(M_1'), \rFs_i(M_2'), \ldots, \rFs_i(M_r'))$ is the $\wpreceq{s_i\underline{w} \hskip 0.1em  \underline{u}} $-cuspidal decomposition of $\rFs_i(M)$.
Moreover, since $- \wt(\rFs_i(M_k')) = -s_i \wt(M_k')$ is larger than any positive root in $\prD \cap s_iw \nrD$ with respect to $\wpreceq{s_i\underline{w} \hskip 0.1em  \underline{u}} $,
$\rFs_i(M)$ is contained in $\catC_{ *, s_iw}$ by Proposition \ref{Prop: membership}.
\end{proof}

\begin{rem}
It was proved in \cite[Corollary 2.26]{TW16}  that the cuspidal decomposition has compatibility with the Saito crystal reflections $T_i$ on simple $R$-modules.
\end{rem}

Thanks to Lemma \ref{Lem: ep and ep*} and Lemma \ref{Lem: Ti and C}, for $w,v\in \weyl$ with $s_iw > w$ and $s_i v > v$, the restrictions of the functors
\begin{align*}
\rF_i|_{\catC_{s_iw, s_iv}} : \catC_{s_i w, s_i v} \longrightarrow \catC_{ w, v}, \qquad
\rFs_i|_{\catC_{w, v}} : \catC_{ w,  v} \longrightarrow \catC_{s_i w, s_i v}
\end{align*}
are well-defined and quasi-inverse to each other by $\eqref{Eq: adjoint}$. Thus, we have the following.

\begin{thm} \label{Thm: equi Cwv}
Let $w, v \in \weyl$ with $v\le w$,   $s_iw > w$ and $s_iv > v$.
Then the restrictions of the functors
\begin{align*}
\rF_i|_{\catC_{s_iw, s_iv}} : \catC_{s_i w, s_i v} \buildrel \sim \over \longrightarrow \catC_{ w, v}, \qquad
\rFs_i|_{\catC_{w, v}} : \catC_{ w,  v} \buildrel \sim \over\longrightarrow \catC_{s_i w, s_i v}
\end{align*}
give equivalences of the categories.
\end{thm}

\begin{cor} Let $w, u, v \in \weyl$ such that $w = vu$ and $\ell(w) = \ell(v) + \ell(u)$. Then there is an equivalence of the categories between
$\catC_{w,v}$ and $\catC_{u}$.
\end{cor}

Let $w, v \in \weyl$ with $w > v$. For a reduced expression $\underline{w}$ of $w$, we define
\begin{equation} \label{eq:Jwv}
J_{\underline{w}, v} \seteq \{  k \in \{1,2, \ldots, \ell(w) \} \mid v_{\le k} = v_{\le k-1} \},
\end{equation}
where $v_{\le k}$ is defined in $\eqref{Eq: def of wk vk}$.
Note that $|J_{\underline{w}, v}| = \ell(w) - \ell(v)$.

Suppose that  $w,v, u\in \weyl$  such that $w = vu$ and $\ell(w) = \ell(v) + \ell(u)$. We fix reduced expressions $\underline{v} = s_{i_1} \cdots s_{i_p}$
and $\underline{u} = s_{i_{p+1}} \cdots s_{i_q} $ of $v$ and $u$ respectively, and set $\underline{w} = \underline{v} \hskip 0.1em  \underline{u} $. Note that (for the notation see \eqref{eq:Jwv})
\[
J_{\underline{w}, v} = \{  p+1, p+2, \ldots,  q  \}.
\]

We define
\[
\rW \seteq \rFs_{i_1}\cdots \rFs_{i_p}: \catC_{u} \buildrel \sim\over \longrightarrow \catC_{w, v}.
\]
Then $\eqref{Eq: refl and crystal}$ gives
\begin{align*}
\rW( \dM( s_{i_{p+1}}\cdots s_{i_{p+k}} \Lambda_{i_{p+k}}, \Lambda_{i_{p+k}} ) ) &\simeq \dM( vs_{i_{p+1}}\cdots s_{i_{p+k}} \Lambda_{i_{p+k}}, v\Lambda_{i_{p+k}} )\\
 &= \dM( w_{\le p+k} \Lambda_{i_{p+k}}, v_{ \le p+k}\Lambda_{i_{p+k}} )
\end{align*}
for $k=1, \ldots, q$.
It was proved in \cite{KKKO17} that   the category $\catC_u$
and the determinantial modules
$$\set{\dM( s_{i_{p+1}}\cdots s_{i_{p+k}} \Lambda_{i_{p+k}}, \Lambda_{i_{p+k}} } {1\le k \le q}$$
 give  a monoidal categorification of  $\Aq[\n(u^{-1})]_{\Z[q,q^{-1}]}$.

\begin{conj} \label{Conj: monoidal categorification}
Let $w,v, u\in \weyl$  such that $w = vu$ and $\ell(w) = \ell(v) + \ell(u)$.
Then,  the category $\catC_{w,v}$ with the determinantial modules $\dM( w_{\le k} \Lambda_{i_k}, v_{ \le k}\Lambda_{i_k} )$ for $k \in J_{\underline{w}, v}$ give
the monoidal categorification of $A_{w,v}=K_0(\catC_{w,v})$  induced from that of $\Aq[\n(u^{-1})]_{\Z[q,q^{-1}]} $ via the functor $\rW$.
\end{conj}

\newcommand{\sodot}{\mathop{\mbox{\normalsize$\bigodot$}}\limits}
\nc{\snconv}{\mbox{\scriptsize$\odot$}}
\begin{rem} \label{Rmk: Conj}
It was shown in \cite[Section 4]{Kato14} that
the exact functor $\rF_i$ is compatible with the convolution product in the case of only PBW modules.
If $\rF_i$ and $\rFs_i$ are monoidal functors, then it
implies  that Conjecture \ref{Conj: monoidal categorification} is true.
Indeed, if $\rW$ is a monoidal functor, then we have a $\Z[q,q^{-1}]$-algebra isomorphism  from $\Aq[\n(u^{-1})]_{\Z[q,q^{-1}]} $ to $A_{w,v}$.
Recall that
$$(\{q^{-(d_k,d_k)/4}  [M_k]\}_{1\le k \le q}, L, {\widetilde B} )$$
gives an initial seed of the quantum cluster algebra $\Q(q^{1/2})\otimes_{\Z[q,q^{-1}]} \Aq[\n(u^{-1})]_{\Z[q,q^{-1}]} $
where
$M_k=\dM( s_{i_{p+1}}\cdots s_{i_{p+k}} \Lambda_{i_{p+k}}, \Lambda_{i_{p+k}})$, \
$d_k=\wt(M_k)$, \  $L=-(\Lambda({M_i,M_j}))_{1\le i,j \le q}$ and  $\tilde B$ is the skew-symmetric matrix given in \cite[Definition 11.1.1]{KKKO17}.
It follows that $ \Q(q^{1/2})\otimes_{\Z[q,q^{-1}]} A_{w,v}  $ is equipped with a quantum cluster algebra structure with an initial seed
$$(\{q^{-(d_k,d_k)/4}  [\rW(M_k)]\}_{1\le k \le q}, L, {\widetilde B} ).$$
Moreover, the pair $(\{M_k\}_{1\le k\le q}, \widetilde{B})$ is \emph{admissible} in the sense
\cite[Definition 7.1.1]{KKKO17}.
   The key condition  is that
 for each mutable index $k$, there exists a  self-dual  simple object $M'_k$ of $\catC_u$
 such that
 there is an exact sequence in $\catC_u$
  \eqn
&&0 \to q \sodot_{b_{ik} >0} M_i^{\snconv b_{ik}} \to q^{\tLa(M_k,M_k')} M_k \conv M_k' \to
 \sodot_{b_{ik} <0} M_i^{\snconv (-b_{ik})} \to 0,
 \eneqn
 where $\sodot$ denotes a grading shift of the usual convolution product which depends only on $L$.
By applying $\rW$, we get an exact sequence in $\catC_{w,v}$
  \eqn
&&0 \to q \sodot_{b_{ik} >0} \rW(M_i)^{\snconv b_{ik}} \to q^{\tLa(\rW(M_k),\rW(M_k'))} \rW(M_k) \conv \rW(M_k') \to
 \sodot_{b_{ik} <0} \rW(M_i)^{\snconv (-b_{ik})} \to 0,
 \eneqn
 since we have $\Lambda(M,N) = \Lambda(\rW(M),\rW(N))$ for any real simple modules $M,N \in \catC_u$.
In turn, the pair in $\catC_{w,v}$
$$\bl  \{ \dM( w_{\le p+k} \Lambda_{i_{p+k}}, v_{ \le p+k}\Lambda_{i_{p+k}} )\}_{1\le k \le q},  {\widetilde B} \br$$
is admissible.
Then, by \cite[Theorem 7.1.3]{KKKO17}, we conclude that
$\catC_{w,v}$ is a monoidal categorification of the quantum cluster algebra $A_{w,v}$.
\end{rem}

\begin{rem}
Let us recall the notations in {\rm Remark \ref{Rmk: geo}}.
Under the assumption on $w$ and $v$ in
{\rm Conjecture \ref{Conj: monoidal categorification}},
the unipotent group $N_{v}^w$ of $N$ corresponding to the positive roots in $\prD \cap w \nrD \cap v \prD$ is a natural fundamental domain
for the action $N'(w) \times N(v)$ on $N$.  As explained in \cite[Section 5]{Lec15}, the coordinate ring $\C[N_{v}^w]$ has a cluster structure
which is isomorphic to that of  $\C[N(u)]$. Thus,
{\rm Conjecture \ref{Conj: monoidal categorification}} tells that
the category $\catC_{w,v}$ is a monoidal categorification of a quantization of the cluster algebra $ \C[N_v^w]$ in terms of quiver Hecke algebras.
\end{rem}


\vskip 2em

\begin{thebibliography}{99}

 \bibitem{BK09}
J.~Brundan and A.~Kleshchev, \emph{Blocks of cyclotomic Hecke algebras and Khovanov-Lauda algebras}, Invent. Math. \textbf{178} (2009), 451--484.




\bibitem{HK02} J.~Hong and S.-J.~Kang,
\emph{Introduction to Quantum Groups and Crystal Bases},
Graduate Studies in Mathematics, \textbf{42}. American Mathematical Society, Providence, RI, 2002.



\bibitem{GLS11} C. Gei\ss, B. Leclerc and J. Schr\"oer,
{\em Cluster structures on quantum coordinate rings},
Selecta Math. (N.S.) {\bf 19}  (2013),  no. 2, 337--397.



\bibitem{KK11}
S.-J. Kang and M. Kashiwara, \emph{Categorification of Highest Weight Modules via Khovanov-Lauda-Rouquier Algebras},
 Invent. Math. \textbf{190} (2012), no. 3, 699--742.



\bibitem{KKK}
S.-J. Kang, M. Kashiwara and M. Kim, {\em Symmetric quiver
Hecke algebras and R-matrices of quantum affine algebras},
Invent. Math. \textbf{211} (2018), no. 2, 591--685.





\bibitem{KKKO15}
S.-J. Kang, M. Kashiwara,  M. Kim  and   S.-j. Oh,
\newblock{\em Simplicity of heads and socles of tensor products},
Compos. Math. \textbf{151} (2015), no. 2, 377--396.




\bibitem{KKKO17}
\bysame,
\newblock{\em Monoidal categorification of cluster algebras}, 
J. Amer. Math. Soc. \textbf{31} (2018), no. 2, 349--426.








\bibitem{K91}
M.~Kashiwara, \newblock{\em On crystal bases of the $q$-analogue of
universal enveloping algebras},
Duke Math. J. {\bf 63} (1991) 465--516.


\bibitem{K93}
\bysame, \newblock{\em Global crystal bases of quantum groups},
Duke Math. J. {\bf 69} (1993), no.\ 2,  455--485.


\bibitem{KP16}
M.~Kashiwara and E.~Park, \newblock{\em Affinizations and $R$-matrices for quiver Hecke algebras}, arXiv:1505.03241;
to appear in Journal of the European Mathematical Society.



\bibitem{Kato14}
S.~Kato, \emph{Poincar\'{e}-Birkhoff-Witt bases and Khovanov-Lauda-Rouquier algebras},
Duke Math.~J.\ \textbf{163} (2014), no. 3, 619--663.


\bibitem{Kato17}
\bysame, \emph{On the monoidality of Saito reflection functors},
arXiv:1711.09085.


\bibitem{KL09}
M.~Khovanov and A. Lauda, \emph{A diagrammatic approach to
categorification of quantum groups
  {I}}, Represent. Theory \textbf{13} (2009), 309--347.

\bibitem{KL11}
\bysame, \emph{A diagrammatic approach to categorification of
  quantum groups {II}}, Trans. Amer. Math. Soc. \textbf{363} (2011),
  2685--2700.



\bibitem{Kimura12}
Y.~Kimura, \emph{Quantum unipotent subgroup and dual canonical basis},
Kyoto J.\ Math.\ \textbf{52} (2012), no.~2, 277--331.



\bibitem{LV09}
A.~Lauda and M.~Vazirani, \emph{Crystals from categorified quantum groups},
  Adv. Math. \textbf{228} (2011), no.~2, 803--861.


\bibitem{Lec15}
B.~Leclerc, \emph{Cluster structures on strata of flag varieties}, 
Adv. Math. \textbf{300} (2016), 190--228. 




\bibitem{LenYak16}
T.~H.~Lenagan and M.~T.~Yakimov, \emph{Prime factors of quantum Schubert cell algebras and clusters for
quantum Richardson varieties},  arXiv:1503.06297; to appear in J. reine angew. Math.
(DOI: https://doi.org/10.1515/crelle-2016-0046)


\bibitem{Mc17}
P.~J.~McNamara, \emph{Monoidality of Kato's Reflection Functors}, arXiv:1712.00173.



\bibitem{MT16}
P.~J.~McNamara and P.~Tingley, \emph{Face functors for KLR algebras},   Represent. Theory \textbf{21} (2017), 106--131.




\bibitem{R08}
R.~Rouquier, \emph{2-Kac-Moody algebras},   arXiv:0812.5023\/v1.

\bibitem{R11}
\bysame, {\em Quiver Hecke algebras and 2-Lie algebras},
Algebra Colloq. {\bf 19} (2012), no. 2, 359--410.


\bibitem{Saito94}
Y.~Saito, \emph{PBW basis of quantized universal enveloping algebras},
Publ.\ Res.\ Inst.\ Math.\ Sci.\ \textbf{30} (1994), no.~2, 209--232.


\bibitem{TW16}
P. Tingley and B. Webster,
{\em Mirkovi\'c-Vilonen polytopes and Khovanov-Lauda-Roquuier algebras},
Compos. Math. 152 (2016), no. 8, 1648--1696.


\bibitem{VV09}
M. Varagnolo and E. Vasserot,
 \emph{Canonical bases and KLR algebras},
J. Reine Angew. Math. \textbf{659} (2011), 67--100.

\end{thebibliography}
\end{document}